\documentclass[a4paper,11pt]{amsart}
\usepackage[left=2.7cm,right=2.7cm,top=3.5cm,bottom=3cm]{geometry}

\usepackage{amsthm,amssymb,amsmath,amsfonts,mathrsfs,amscd}
\usepackage[latin1]{inputenc}
\usepackage[all]{xy}
\usepackage{latexsym}
\usepackage{longtable}

\newfont{\cyr}{wncyr10 scaled 1100}
\newfont{\cyrr}{wncyr9 scaled 1000}

\theoremstyle{plain}
\newtheorem{theorem}{Theorem}[section]

\newtheorem{lemma}[theorem]{Lemma}

\theoremstyle{definition}
\newtheorem{conjecture}[theorem]{Conjecture}
\newtheorem{definition}[theorem]{Definition}

\theoremstyle{remark}
\newtheorem{remark}[theorem]{Remark}

\newtheorem*{acknowledgements}{Acknowledgements}

\newcommand{\Q}{\mathbb Q}
\newcommand{\N}{\mathbb N}
\newcommand{\Z}{\mathbb Z}
\newcommand{\R}{\mathbb R}
\newcommand{\C}{\mathbb C}
\newcommand{\A}{\mathbb A}

\DeclareMathOperator{\Pic}{Pic}

\DeclareMathOperator{\Emb}{Emb}

\DeclareMathOperator{\Hom}{Hom}

\DeclareMathOperator{\Gal}{Gal}
\DeclareMathOperator{\GL}{GL}

\DeclareMathOperator{\SL}{SL}

\DeclareMathOperator{\M}{M}

\DeclareMathOperator{\vol}{vol}

\newcommand{\ord}{\mathrm{ord}}

\newcommand{\bs}{\backslash}
\newcommand{\bmx}{\left( \begin{matrix}}
\newcommand{\emx}{\end{matrix} \right)}
\newcommand{\leg}{\overwithdelims ()}

\usepackage[usenames]{color}
\definecolor{Indigo}{rgb}{0.2,0.1,0.7}
\definecolor{Violet}{rgb}{0.5,0.1,0.7}
\definecolor{White}{rgb}{1,1,1}
\definecolor{Green}{rgb}{0.1,0.9,0.2}

\newcommand{\mat}[4]{\left(\begin{array}{cc}#1&#2\\#3&#4\end{array}\right)}
\newcommand{\smallmat}[4]{\bigl(\begin{smallmatrix}#1&#2\\#3&#4\end{smallmatrix}\bigr)}

\usepackage{bm}
\def \mint {\boldsymbol{\times} \hskip -1.138em \int}
\def \smint {\boldsymbol{\times} \hskip -0.96em \int}

\newcommand{\cO}{{\mathcal O}}
\newcommand{\E}{{\mathcal E}}

\include{thebibliography}

\begin{document}

\title[]{Rationality of Darmon points over genus fields of non-maximal orders}
\author{Matteo Longo, Kimball Martin and Yan Hu}
\thanks{}

\begin{abstract} Stark--Heegner points, also known as \emph{Darmon points}, 
were introduced by H. Darmon in \cite{Dar} 
as certain local points on rational elliptic curves, conjecturally defined over abelian extensions 
of real quadratic fields. The rationality conjecture for these points is only known in the unramified case, namely, when these points 
are specializations of global points defined over the strict Hilbert class field $H^+_F$ 
of the real quadratic field $F$ and twisted by (unramified) quadratic 
characters of $\Gal(H_c^+/F)$.  
We extend these results to the situation of 
ramified quadratic characters; more precisely, we show that
Darmon points of conductor $c\geq 1$ twisted by quadratic characters of $G_c^+=\Gal(H_c^+/F)$, where
$H_c^+$ is the strict ring class field of $F$ of conductor $c$, come from rational points on the elliptic curve defined over $H_c^+$. 

\bigskip 
Les points de Stark-Heegner, \'egalement connus sous le nom de \emph{points de Darmon},
ont \'et\'e introduits par H. Darmon dans \cite{Dar}
comme certains points locaux sur des courbes elliptiques rationnelles. 
Les points de Stark-Heegner sont conjecturalement d\'efinies sur des extensions ab\'eliennes
d'un corps quadratiques r\'eel. La conjecture de rationalit\'e pour ces points est connue dans le cas non ramifi\'e. Dans ce cas, ces points
sont des sp\'ecialisations de points globaux d\'efinies sur le corps de classes 
de Hilbert au sens strict $H ^+_ F$ du corps quadratique r\'eel $F$, 
et tordu par le caract\`eres quadratique (non ramifi\'e)
de $\Gal (H_c ^ + / F) $.
Nous \'etendons ces r\'esultats \`a la situation de
caract\`eres quadratiques ramifi\'es; plus pr\'ecis\'ement, nous montrons que
les points de Darmon de conducteur $ c \geq 1 $ tordus par des caract\`eres quadratiques de $ G_c ^ + = \Gal (H_c ^ + / F) $, o\`u
$ H_c ^ + $ est le corp de classe de $ F $ au sens strict du conducteur $ c $, provenant de points d\'efinie sur $ H_c ^ + $ de la courbe elliptique.
\end{abstract}

\address{Dipartimento di Matematica, Universit\`a di Padova, Via Trieste 63, 35121 Padova, Italy}
\email{mlongo@math.unipd.it}
\email{sthuyan@gmail.com}

\address{Department of Mathematics, University of Oklahoma, Norman, OK 73072
USA}
\email{kmartin@math.ou.edu}

\subjclass[2000]{}
\date{\today}

\maketitle
\tableofcontents

\section{Introduction} 

The theory of \emph{Stark--Heegner points}, also known as \emph{Darmon points},
began with the foundational 
paper by H. Darmon \cite{Dar} in 2001. In this work, Darmon proposed a construction of local points on rational 
elliptic curves, the Stark--Heegner points, which, under appropriate arithmetic conditions, he conjectured to be global points defined over strict ring class fields of \emph{real quadratic fields}, which are non-torsion when the central critical value of the 
first derivative of the complex $L$-function of the elliptic curve over the real quadratic field does not vanish. 
Note that the absence of a theory of complex multiplication in the real quadratic case, available in the  imaginary quadratic case, makes the construction of 
global points on elliptic curves over real quadratic fields and their abelian extensions a rather challenging problem. 
The idea of Darmon was to define locally a family of candidates for their points, and conjecture that these come from global points. Following \cite{Dar}, many authors proposed similar constructions in different situations, including the cases of modular and Shimura curves, and the higher weight analogue of Stark--Heegner, or 
Darmon, cycles; 
with no attempt to be complete, see for instance \cite{Das}, \cite{MG}, \cite{LRV1}, \cite{LRV}, \cite{Trif}, \cite{GS}, 
\cite{RS}, \cite{GM1}, \cite{GM2}, \cite{GM3}, \cite{GM4}, and \cite{GM5}.   
The relation between Stark--Heegner points and half-integral weight modular forms is studied 
in \cite{DT}, \cite{LZ}, \cite{LN}, \cite{Guhan}. 

The arithmetic setting of the original construction in \cite{Dar} is the following. Fix a
rational elliptic curve 
$E$ of conductor $N=Mp$, with $p\nmid M$ an odd prime number and $M \geq1$ an integer. 
Fix also a 
real quadratic field $F$ satisfying the following \emph{Stark--Heegner assumption}: all primes $\ell\mid M$ 
are split in $F$, while $p$ is inert in $F$.
Under these assumptions, the central critical value $L(E/F,1)$ 
of the complex $L$-function of $E$ over $F$ vanishes. Darmon points 
are local points $P_c$ for $E$ defined of finite extension of $F_p$, the completion of $F$ 
at the unique prime above $p$; their definition and the main properties are recalled in Section \ref{sec1} below. 
The definition of these points depends on the choice of an auxiliary integer $c\geq1$, 
called the \emph{conductor} of a Darmon point $P_c$.  The rationality conjecture 
predicts that these points $P_c$ are localizations of global points $\mathbf{P}_c$ which 
are defined over the strict ring class field $H_c^+$ of $F$ of conductor $c$.

Only partial results are known toward the rationality conjectures for Darmon points, or more generally cycles. 
The first result on the rationality of Darmon points 
is due to Bertolini and Darmon in the paper \cite{BD-Rationality}, where they show that 
a certain linear combination of these points with coefficients given by values of genus characters 
of the real quadratic field $F$ comes from a global point defined over the Hilbert class field of $F$.  
The main idea behind the proof of these results is to use a factorization formula for $p$-adic 
$L$-functions to compare the localization of Heegner points and Darmon points.  
More precisely, the first step of the proof consists in relating Darmon points 
to the $p$-adic $L$-function interpolating central critical values of the complex  
$L$-functions over $F$ attached to the arithmetic specializations of the Hida family 
passing through the modular form attached to $E$.  The second step consists in expressing 
this $p$-adic $L$-function
in terms of a product of two Mazur--Kitagawa $p$-adic $L$-functions, which are known to be 
related to Heegner points by the main result of \cite{BD-HidaFamilies}. 
A similar strategy has been adopted by \cite{GSS}, \cite{S} \cite{LV}, \cite{LV1} obtaining similar results. 

All known results in the direction of the conjectures in \cite{Dar} involve linear combination of Darmon
points twisted by genus characters, which are 
quadratic unramified characters of $\Gal(H^+_F/F)$, where $H_F^+$ is the 
(strict) Hilbert class field of $F$.  
The goal of this paper is to prove a similar rationality result for more general quadratic characters, 
namely, quadratic characters of ring class fields of $F$, so we allow for ramification. 
In the remaining part of the introduction we briefly state our main result
and the main differences with the case of genus characters treated up to now. 

Let $E/\Q$ be an elliptic curve, and denote by $N$ its conductor. 
Let $F/\Q$ be a real quadratic field $F=\Q(\sqrt{D})$ 
of discriminant $D=D_F>0$, prime to $N$. 
We assume 
one has a factorization 
$N=Mp$ with $p\nmid M$, such that  
all primes $\ell\mid M$ are split in $F$ and $p$ is inert in $F$. 

Fix an integer $c$ prime to $D\cdot N$ 
and a quadratic character \[\chi:G_c^+=\Gal(H_c^+/F)\longrightarrow\{\pm1\},\] where, as above, 
$H_c^+$ denotes the strict class field of $F$ of conductor $c$.  Let $\mathcal O_c$
be the order in $F$ of conductor $c$.  Recall that 
$G_c^+$ is isomorphic to the group of strict equivalence classes of projective 
$\mathcal{O}_c^+$-modules, which we denote $\Pic^+(\mathcal{O}_c)$, 
where two such modules are 
strictly equivalent if they are the same up to an element of $F$ of positive norm. 
We assume that $\chi$ is \emph{primitive}, meaning that it does not factor through any 
$G_f^+$ with $f$ a proper divisor of $c$. 

Fix  embeddings $F\hookrightarrow\bar\Q$ 
and $\bar\Q\hookrightarrow\bar\Q_p$ throughout. Let $P_c\in E(F_p)$ be a Darmon point 
of conductor $c$ (see Section \ref{sec1} below for the precise definition of these points) 
where, as above, $F_p$ is the completion of $F$ at the unique prime of $F$ above $p$. 
It follows from their construction that Darmon points of conductor $c$ 
are in bijection with equivalence classes of quadratic forms of discriminant $Dc^2$, and this can 
be used to \emph{define} a Galois action $P_c\mapsto P_c^\sigma$ on  Darmon points,
where $P_c$ is a fixed Darmon point of conductor $c$ and $\sigma\in G_c^+$. 
We may then form the point 
\begin{equation} \label{Pchi}
P_\chi=\sum_{\sigma\in G_c^+} \chi^{-1}(\sigma)P_c^\sigma
\end{equation}
which lives in $E(F_p)$. 
Finally, let $\log_E : E(\C_p)\rightarrow \C_p$ denote the formal group logarithm of $E$. 
Note that, since $p$ is inert in $F$, it splits completely in $H_c^+$, and therefore 
for any point $Q\in E(H_c^+)$ the localization of $Q$ at any of the primes in $H_c^+$ 
dividing $p$ lives in $E(F_p)$. 
Our main result is the following: 

\begin{theorem} Assume that 
$c$ is is odd and coprime to $DN$. Let $\chi$ be a primitive quadratic character of $G_c^+$. 
Then 
there exists a point $\mathbf{P}_\chi$ in $E(H_c^+)$  
and a rational number $n\in\Q^\times$ such that
\[\log_E(P_\chi)=n\cdot\log_E(\mathbf{P}_\chi).\] 
Moreover, the point $\mathbf{P}_\chi$ is of infinite order 
if and only if $L'(E/F,\chi,1)\neq 0$. 
\end{theorem}

If $c=1$, this is essentially the main result of \cite{BD-Rationality}.
To be more precise, the work \cite{BD-Rationality} needed to assume $E$ had two primes of
multiplicative reduction because of this assumption in \cite{BD-HidaFamilies}.
However, this assumption has been removed by very recent work of Mok \cite{Mok1},
which we also apply here.
 
The proof in the general case follows a similar line to that in \cite{BD-Rationality}.
However, some modifications are in order. The first difference is that the 
genus theory of non-maximal orders is more complicated than the usual genus 
theory, and the arguments need to be adapted accordingly. More 
importantly, one of the main ingredients in the proof of 
the rationality result in \cite{BD-Rationality} is a formula of Popa 
\cite{Popa} for the central critical value of the $L$-function 
over $F$ of the specializations at arithmetic points of the Hida family passing through the modular form associated with the elliptic curve $E$. 
However, this formula does not allow to treat $L$-functions twisted by ramified characters.
Instead, we recast an $L$-value formula from \cite{MW} which allows for ramification,
expressed in terms of periods of Gross--Prasad test vectors, in a more classical
framework to get our result.

\begin{remark} Let $\pi_f$ and $\pi_\chi$ be the automorphic representations of $\GL_2(\mathbb{A}_\Q)$ attached to $f$ and $\chi$, respectively, 
where $\mathbb{A}_\Q$ is the adele ring of $\Q$. 
When $(cD, N)\neq 1$,
it may be that $\pi_f$ and $\pi_\chi$ have joint ramification.
In this case, we can instead use \cite{FMP} in lieu of \cite{MW}, 
at least in the case that $f$ has squarefree level.
\end{remark} 

\begin{remark}
The main result of this paper assumes that all primes dividing $M$ are split
and $p$ is inert in $F$. More generally, the same results are expected to hold 
under the following \emph{relaxed modified Heegner assumption}: 
$p$ is inert in $F$ and there is a factorization $M=M^+\cdot M^-$ 
of $M$ into coprime integers such that a prime number $\ell\mid M^+$ if and only if $\ell$ in 
split in $F$, and $M^-$ is a product of an even number of distinct primes. 
If the conductor $N$ of the elliptic curve $E$ can be factorized as $N=M^+\cdot M^-\cdot p$ 
with $p\nmid M$ and the discriminant $D$ of the real quadratic field $F$, 
and $\chi$ is a primitive quadratic character of $G_c^+$ with $c$ odd and coprime with 
$ND$, then one can show that there exists a point $\mathbf{P}_\chi$ in $E(H_c^+)$  
and a rational number $n\in\Q^\times$ such that
$\log_E(P_\chi)=n\cdot\log_E(\mathbf{P}_\chi)$; 
moreover, the point $\mathbf{P}_\chi$ is of infinite order 
if and only if $L'(E/F,\chi,1)\neq 0$. 
The proof of this result can be obtained with the methods of this 
paper by replacing modular curves of level $M$ with Shimura curves 
of  level $M^+$ attached to quaternion algebras of discriminant $M^-$, 
following what is done in \cite{LV} in the case $c=1$. However, since the notation 
in the quaternionic case is quite different from the notation in the case of modular curves, 
we prefer to only treat in detail the case when $M^-=1$. Details in the case $M^->1$ 
are left to the reader. 
\end{remark} 

\begin{remark} 
The main result of this paper plays a role in the forthcoming works 
by Darmon-Rotger \cite{DR} and Bertolini-Seveso-Venerucci \cite{BSV} 
proving the rationality of Darmon points (actually, cohomology classes closely related 
to Darmon points) in situations which go far beyond the case of quadratic 
characters considered in this paper. This application 
was one of the main motivations for this work. 
\end{remark}

\begin{acknowledgements} 
We are grateful to H. Darmon for suggesting the 
problem and V. Rotger
for many interesting discussions about the topics of this paper.
We also thank C.P. Mok for providing us with a preliminary version of
\cite{Mok1}.
M.L. is supported by PRIN  2015, INdAM--GNASGA, BIRD 2017. 
K.M. was supported by a grant from the Simons 
Foundation/SFARI (512927, KM).
\end{acknowledgements}

\section{Darmon points} \label{sec1}

Let the notation be as in the introduction: $E/\Q$ is an elliptic curve of conductor $N=Mp$ 
with $p\nmid M$, and $F/\Q$ is a real quadratic field of discriminant $D=D_F$ such that 
all primes dividing $M$ are split in $F$ and $p$ is inert in $F$. Finally, $c$ is a positive 
integer prime to $ND$.  
The aim of this section is to review the definition of Darmon points and 
some results in \cite{BD-Rationality} and \cite{BD-HidaFamilies}. 

We first set up some standard notation. For any field $L$, let $P_{k-2}(L)$ be the space of 
homogeneous polynomials in $2$ variables of degree $k-2$, and let 
$V_{k-2}(L)$ be its $L$-linear dual. We let $\gamma=\smallmat abcd\in\GL_2(L)$ 
act from the right on $P(x,y)\in P_{k-2}(L)$ by the formula \[(P|\gamma)(x,y)=P(ax+by,cx+dy)\]
and we equip $V_{k-2}(L)$ with the dual action.  
 If $G$ is any abelian group, let $\mathrm{MS}(G)$ 
be the group of $G$-valued modular symbols, {i.e.}\ the $\Z$-module 
of functions $I:\mathbb{P}^1(\Q)\times\mathbb{P}^1(\Q)\rightarrow G$ such that 
$I(x,y)+I(y,z)=I(x,z)$ for all $x,y,z\in\mathbb{P}^1(\Q)$. The value $I(r,s)$ of $I\in \mathrm{MS}(G)$ 
on $(x,y)$ will be usually denoted $I\{x\rightarrow y\}$. 
The group $\GL_2(\Q)$ 
acts from the left by fractional linear transformations on $\mathbb{P}^1(\Q)$, 
and if $G$ is equipped with a left $\mathbb{P}^1(\Q)$-action, then 
$\mathrm{MS}(G)$ inherits a right $\GL_2(\Q)$-action by the rule $(I|\gamma)(x,y)=
\gamma\cdot I(\gamma^{-1}x,\gamma^{-1}y)$. If $\Gamma_0$ is a subgroup of 
$\GL_2(L)$, we denote $\mathrm{MS}_{\Gamma_0}(G)$ the subgroup of those 
elements in $\mathrm{MS}(G)$ which are invariant under the action of $\gamma$ 
for all $\gamma\in\Gamma_0$. If $f$ is a cuspform of level $\Gamma_0(M)$ and weight $k$, 
we may attach to $f$ the standard modular symbol 
$\tilde I_{f}\in\mathrm{MS}_{\Gamma_0(M)}(V_{k-2}(\C))$; 
explicitly, for 
$r,s\in\mathbb{P}^1(\Q)$ and $P(x,y)\in P_{k-2}(\C)$ an homogenous polynomial 
of degree $k-2$, put 
\[\tilde I_{f}\{r\rightarrow s\}(P(x,y))=2\pi i \int_r^s f(z)P(z,1) \, dz.\]  
The matrix $\omega_\infty=\smallmat 100{-1}$ acts on the group of modular symbols 
$\mathrm{MS}_{\Gamma_0(M)}(V_{k-2}(\C))$, 
and we let $\tilde I_f^\pm$ denote the projections to the $\pm$-eigenspaces for this action. 
Suppose that $f$ is a normalized eigenform
and let $K_{f}$ be the field generated over $\Q$ by the Fourier coefficients of 
$f$. Then there are complex periods $\Omega_{f}^\pm$ 
for each choice of sign $\pm$ 
such that their product equals the Petersson inner product $\langle f,f\rangle$, and  
$I_{f}^\pm:=\tilde I_f^\pm/\Omega_{f}^\pm$ satisfies the condition that 
if $P(x,y)\in P_{k-2}(K_{f})$ then $I_{f}\{r\rightarrow s\}(P(x,y))$ belongs to 
$K_{f}$. 

\subsection{Measure-valued modular symbols and Darmon points}
Let $f$ be the newform of level $N$ attached to $E$ by modularity. 
Denote $B=\M_2(\Q)$ the split quaternion algebra over $\Q$ 
and let $R$ be the $\Z[1/p]$-order in $B$ consisting of matrices in $\M_2(\Z[1/p])$ 
which are upper triangular modulo $M$. Define 
\[\Gamma=\{\gamma\in R^\times \ | \ \det(\gamma)=1\}.\] 
Let $\mathrm{Meas}^0(\mathbb{P}^1(\Q_p),\Z)$ denote the $\Z$-module 
of $\Z$-valued measures on $\mathbb{P}^1(\Q_p)$ with total 
measure equal to $0$. 
By \cite[Proposition 1.3]{BD-Rationality}, for each choice of sign 
$\pm$, there exists a unique
function, which we call the \emph{measure-valued modular symbol} 
attached to $f$, 
\[\mu_f^\pm:\mathbb{P}^1(\Q)\times\mathbb{P}^1(\Q)\longrightarrow
\mathrm{Meas}^0\left(\mathbb{P}^1(\Q_p),\Z\right)\]
denoted $(r,s)\mapsto \mu_f\{r\rightarrow s\}$, satisfying the 
following conditions: 
\begin{enumerate}
\item $\mu_f^\pm\{r\rightarrow s\}(\Z_p)=I_f^\pm\{r\rightarrow s\}$ 
\item For all $\gamma\in\Gamma$ and all open compact subsets $U\subseteq\mathbb{P}^1(\Q_p)$, 
\[\mu_f\{\gamma(r)\rightarrow\gamma(s)\}(U)=\mu_f\{r\rightarrow s\}(U),\] 
where we let $\GL_2(\Q_p)$ act on $\mathbb{P}^1(\Q_p)$ 
by fractional linear transformations. 
\end{enumerate} 

Let $\mathcal{H}_p=\C_p\setminus\Q_p$ denote the $p$-adic upper half plane.
The system of measures $\mu_f$ can be used to define, 
for any $r,s\in\mathbb{P}^1(\Q)$ and $\tau_1,\tau_2\in\mathcal{H}_p$, 
a \emph{double multiplicative integral}
\[\mint_{\tau_1}^{\tau_2}\int_r^s \omega_f:= \mint_{\mathbb{P}^1(\Q_p)}
\frac{t-\tau_2}{t-\tau_1} \, d\mu_f\{r\rightarrow s\}(t).
 \] 

 (On the right, the notation $\smint$ refers to the fact that the integration 
 is relative to the multiplicative structure of $\C_p^\times$, and therefore 
 is a limit of Riemann products instead of Riemann sums.) Let $q$ be the Tate period of $E$ at $p$, and let $\log_q$ be the branch of the $p$-adic logarithm 
satisfying $\log_q(q)=0$. Define the additive version of the double 
multiplicative integral to be 
\[\int_{\tau_1}^{\tau_2}\int_r^s \omega_f:=
\log_q\left(\mint_{\tau_1}^{\tau_2}\int_r^s \omega_f\right).\]

We finally introduce the notion of \emph{indefinite integral}. 
By \cite[Proposition 1.5]{BD-Rationality}, there exists a unique 
function from $\mathcal{H}_p\times\mathbb{P}^1(\Q)\times\mathbb{P}^1(\Q)$ 
to $\C$, denoted $(\tau,r,s)\rightarrow\int^\tau\int_r^s\omega_f$, satisfying the 
following conditions: 
\begin{enumerate}
\item The integral is $\Gamma$-invariant, in the sense that  
for all $\gamma\in\Gamma$, we have 
\[\int^{\gamma(\tau)}\int_{\gamma(r)}^{\gamma(s)}\omega_f=\int^\tau\int_r^s\omega_f\]
\item For any pair $\tau_1,\tau_2\in\mathcal{H}_p$, we have 
\[\int^{\tau_2}\int_r^s\omega_f-\int^{\tau_1}\int_r^s\omega_f=\int_{\tau_1}^{\tau_2}
\int_r^s\omega_f;\] 
\item For all $r,s,t$ in $\mathbb{P}^1(\Q)$ we have 
\[\int^\tau\int_r^s\omega_f+\int^\tau\int_s^t\omega_f=\int^\tau\int_r^t\omega_f.\] 
\end{enumerate}

We now define 
Darmon points using indefinite integrals above. 
Since $p$ is inert in $F$, the set $\mathcal{H}_p\cap F$ is not empty and one may define the 
order $\mathcal O_\tau$ associated with $\tau\in \mathcal{H}_p\cap F$ as 
\[\mathcal{O}_\tau=\left\{\mat abcd \in R \ | \ a\tau+b=c\tau^2+d\tau\right\}.\]
The map $\smallmat abcd\rightarrow c\tau+d$ induces an embedding $\mathcal{O}_\tau\hookrightarrow F$, 
and thus $\mathcal{O}_\tau$ may be viewed as an order in $F$. For $\tau\in \mathcal{H}_p\cap F$, let $\gamma_\tau=\smallmat abcd$ 
denote the unique generator of the stabilizer of $\tau$ in $\Gamma$ satisfying $c\tau+d>1$ 
(with respect to the chosen embedding $F\subseteq\bar\Q$). Let 
$\Phi_\mathrm{Tate}:\C_p^\times/q^\Z\rightarrow E(\C_p)$ denote the 
Tate uniformization of $E$ at $p$. Attached to $\tau$, 
there is an indefinite integral 
\begin{equation}\label{def:J}
J_\tau^\times=\mint^\tau\int_r^{\gamma_\tau(r)}\omega_f\in\C_p\end{equation}
where $r\in\mathbb{P}^1(\Q_p)$ is any base point, and one 
can show that $\Phi_\mathrm{Tate}(J_\tau^\times)$ is a well-defined point in $E(\C_p)$  
independently of the choice of $r$, up to its torsion subgroup 
$E(\C_p)_\mathrm{tors}$.  
\begin{definition}\label{def:Darmon} Let $\tau\in\mathcal{H}_p\cap F$. 
The point $P_\tau=\Phi_\mathrm{Tate}(J_\tau^\times)\in E(\C_p)\otimes_\Z\Q$,  
with $J_\tau^\times$ defined in \eqref{def:J}, 
is the \emph{Darmon point} attached to $\tau$.\end{definition}  

\subsection{Shimura reciprocity law}  
Fix an integer $c$ prime to $D\cdot N$, and let $\mathcal{O}_c$ be the order of 
$F$ of conductor $c$.  Denote $\mathcal Q_{Dc^2}$ the set  
of primitive binary quadratic forms of discriminant $Dc^2$.
Let $\SL_2(\Z)$ act from the right on the set $\mathcal{Q}_{Dc^2}$ 
via the formula 
\begin{equation}\label{action}
(Q|\gamma)(x,y)=Q(ax+by,cx+dy)
\end{equation} for $Q\in \mathcal{Q}_{Dc^2}$ and 
$\gamma=\smallmat abcd$. The set of equivalence classes $\mathcal{Q}_{Dc^2}/\SL_2(\Z)$
is equipped with a group structure given by the Gaussian composition law. 
If $H_c^+$ is the strict ring class field of $F$ of conductor $c$, then its Galois group  
$G_c^+=\Gal(H_c^+/F)$ is isomorphic to the group 
$\mathcal{Q}_{Dc^2}/\SL_2(\Z)$ via global class field theory (see \cite[Theorem 14.19]{cohn}). 

Fix $\delta\in\Z$ such that $\delta^2\equiv D\pmod{4M}$. 
Let $\mathcal{F}_{Dc^2}$ denote the subset of $\mathcal{Q}_{Dc^2}$ consisting 
of forms $Q(x,y)=Ax^2+Bxy+Cy^2$ such that $M\mid A$ and $B\equiv\delta\pmod{2M}$.
The group $\Gamma_0(M)$ acts on $\mathcal{F}_{Dc^2}$ by the formula \eqref{action}.  
Since $(M,D)=1$, we also have $(\delta,M)=1$, and therefore, by \cite[Proposition, p.\ 505]{GKZ}, 
the map $Q\mapsto Q$ sets up a bijection between $\mathcal{F}_{Dc^2}/\Gamma_0(M)$ 
and $\mathcal{Q}_{Dc^2}/\SL_2(\Z)$. In particular, 
the set $\mathcal{F}_{Dc^2}/\Gamma_0(M)$ is equipped with a 
structure of principal homogeneous space under $G_c^+$. If 
$Q\in \mathcal{F}_{Dc^2}/\Gamma_0(M)$ and $\sigma\in G_c^+$, we 
denote $Q^\sigma$ for the image of $Q$ by $\sigma$. 

Define 
\[\mathcal{H}_p^{(Dc^2)}=\{\tau\in\mathcal{H}_p\cap F \ | \ \mathcal{O}_\tau=\mathcal{O}_c\}.\] 
Given $Q(x,y) =Ax^2+Bxy+Cy^2$ a quadratic form in $\mathcal{F}_{Dc^2}$, let $\tau_Q=\frac{-B+c\sqrt{D}}{2A}$ 
be a fixed root of the quadratic polynomial $Q(x,1)$. Then $\tau_Q$ belongs to 
$\mathcal{H}_p^{(Dc^2)}$ (via the fixed $p$-adic embedding of $F$ into $\bar\Q_p$) 
and its image in $\Gamma\backslash\mathcal{H}_p^{(Dc^2)}$ does not 
depend on the $\Gamma_0(M)$-equivalence class of $Q$. Given $\sigma\in G_c^+$, 
we will sometimes write $\tau_{Q}^\sigma$ for $\tau_{Q^\sigma}$.

Let $P$ be a point in $E(H_c^+)$. Since $p$ is inert in $K$, it splits completely in $H_c^+$, 
and therefore, after fixing a prime of $H_c^+$ above $p$, the point $P$ 
localizes to a point in $E(F_p)$, where $F_p$ is the completion of $F$ 
at the unique prime above $p$. 

\begin{conjecture}
The Darmon point $P_{\tau_Q}$ is the localization of 
a global point $P_c$, defined over $H_c^+$, and the Galois action 
on this point is described by the following Shimura reciprocity law: 
if $P_c\in E(H_c^+)$ localizes to $P_{\tau_Q}\in E(F_p)$ then 
$P_c^\sigma$ localizes to $P_{\tau_Q^\sigma}$.
\end{conjecture}

\subsection{Real conjugation}\label{sec:2.3} 
Denote by $\tau_p\in \Gal(H_c^+/\Q)$ 
the Frobenius element at $p$, well defined only up to conjugation. 
As recalled above, since $p$ is inert in $F$, it splits completely in $H_c^+$ 
and (after fixing as above a prime of $H_c^+$ 
above $p$), $\tau_p$ corresponds to an involution of $\Gal(F_p/\Q_p)$. 
By \cite[Proposition 5.10]{Dar}, it is known that there 
exists an element $\sigma_\tau\in G_c^+$ such that
\begin{equation}
\label{real-conj}
\tau_p(J_\tau)=-w_MJ_{\tau^{\sigma_\tau}}\end{equation} and 
 $\tau_p(P_\tau)=w_NP_{\tau^{\sigma_\tau}}$ where $w_M$ and $w_N$ are the signs of the 
Atkin-Lehner involution $W_M$ and $W_N$, respectively, 
acting on $f$. 

\subsection{Families of measure-valued modular symbols and Darmon points} 
\label{sec:2.4}
Let \[\mathcal{X}=\Hom_{\Z_p}^\mathrm{cont}(\Z_p^\times,\Z_p^\times)\] and embed 
$\Z$ into $\mathcal{X}$ via $k\mapsto[x\mapsto x^{k-2}]$. There are 
rigid analytic functions $\kappa\mapsto a_n(\kappa)$ for integers $n\geq 1$, simultaneously 
defined on a suitable neighborhood $\mathcal{U}$ of $2$ (which we may assume
containing only integers $k$ with $k\equiv 2\mod {p-1}$) such that the formal power series expansion
\[f_\infty(\kappa)=\sum_{n\geq 1}a_n(\kappa)q^n\] when evaluated at $\kappa=k\in\Z$ with $k\geq 2$, is the 
$q$-expansion of a  normalized 
eigenform on $\Gamma_0(N)$ of weight $k$, and such that 
$f_2=f$, where recall that $f$ is the modular form attached to $E$ by modularity. 
If $k\neq 2$, $f_k$ is necessarily old, and we let $f_k^\sharp$ be the form of level $\Gamma_0(M)$ and weight $k$ whose $p$-stabilization
coincides with $f_k$; so $f_k$ and $f_k^\sharp$ are related by the formula:
\[f_k(z)=f_k^\sharp(z)-p^{k-1}a_p(k)^{-1}f_k^\sharp(pz).\]  
For $k=2$ we simply put $f_2^\sharp=f_2^{\phantom{\sharp}} =f$. 

Let $\mathcal{W}=\Q_p^2-\{(0,0)\}$ and let $\mathbb{D}$ denote the $\Q_p$-vector space of compactly supported 
$\Q_p$-valued measures on $\mathcal{W}$. Let $L_*=\Z_p^2$. 
Say that $(x,y)\in L_*$ is 
\emph{primitive} if $p$ does not divide both $x$ and $y$ and 
let $L_*'=(\Z_p^2)'$ denote the subset of $L_*$ consisting of primitive vectors.
Let $\mathbb{D}_*$ 
be the subspace consisting of measures which are supported on the 
$L_*'$. 
Let $\Lambda=\Z_p[\![\Z_p^\times]\!]$ be the Iwasawa algebra of $\Z_p^\times$, 
identified with a 
subring of the ring of analytic functions on $\mathcal{X}$. The $\Q_p$-vector 
space $\mathbb{D}$ is equipped with a structure of $\Lambda$-algebra arising form the 
action of $\Z_p^\times$ on $\mathcal{W}$ and $\Z_p^2$ given by $(x,y)\mapsto(\lambda x,\lambda y)$ 
for $\lambda\in\Lambda$. Also, $\GL_2(\Q_p)$ acts from the left on 
$\mathbb{D}_*$ by translations, and $\mathrm{MS}_{\Gamma_0(M)}(\mathbb{D}_*)$ 
is naturally equipped with an action of Hecke operators. 
In particular, we have a $U_p$-operator acting on 
$\mathrm{MS}_{\Gamma_0(M)}(\mathbb{D}_*)$ by the formula 
\[\int_X\phi \, d(U_p\mu)\{r\rightarrow s\}=\sum_{a=0}^{p-1}\int_{p^{-1}\gamma_a X}(\phi|p\gamma_a^{-1}) \,
d\mu\{\gamma_a(r)\rightarrow\gamma_a(s)\}\] 
for any locally constant function $\phi$ on $\mathcal{W}$.
Here $\gamma_a=\smallmat 1a0p$, and for any open compact 
subset $X\subseteq\mathcal{W}$ and any locally constant function $\phi$ on $\mathcal{W}$, 
we put $\int_X\phi \, d\mu=\int_{L'_*}\phi(x)\mathrm{char}_X(x) \, d\mu(x)$,
where $\mathrm{char}_X$  is the characteristic function of $X$.
In particular, 
we may define $\mathrm{MS}_{\Gamma_0(M)}^\ord(\mathbb{D}_*)$ to be the 
maximal submodule of $\mathrm{MS}_{\Gamma_0(M)}(\mathbb{D}_*)$
on which $U_p$ acts invertibly. 
For each $k\in\mathcal{U}$ 
there is a specialization map 
\[\rho_k: \mathbb{D}_*^\dagger\longrightarrow V_{k-2}(\C_p)\]
defined by 
\[\rho_k(\mu)(P(x,y))=\int_{\Z_p\times\Z_p^\times}P(x,y) \, d\mu(x,y).\]

Let $\Lambda^\dagger$ denote the ring of $\C_p$-valued functions on 
$\mathcal{X}$ which can be represented by a convergent power series 
expansion in some neighborhood of $2\in\mathcal{X}$ and 
define $\mathbb{D}_*^\dagger=\mathbb{D}_*\otimes_\Lambda\Lambda^\dagger$. 
For any $\mu=\sum_i\lambda_i\mu_i$ with $\lambda_i\in\Lambda^\dagger$ 
and $\mu_i\in\mathbb{D}_*$, we call a \emph{neighborhood of regularity} 
for $\mu$ any neighborhood $U_\mu$ 
of $2$ in $\mathcal U$ such that all $\lambda_i$ converge 
in $U_\mu$. The module $\mathrm{MS}_{\Gamma_0(M)}^\ord(\mathbb{D}_*)$ inherits 
a $\Lambda$-action from the $\Lambda$-module structure of $\mathbb{D}_*$, and we may 
define
$\mathrm{MS}_{\Gamma_0(M)}^{\ord,\dagger}(\mathbb{D}_*)=
\mathrm{MS}_{\Gamma_0(M)}^\ord(\mathbb{D}_*)\otimes_\Lambda\Lambda^\dagger$.
This $\Lambda^\dagger$-module is free of finite rank, and given 
$\mu\in \mathrm{MS}_{\Gamma_0(M)}^{\ord,\dagger}(\mathbb{D}_*)$ it is possible 
to find a common neighborhood of regularity for all the measures $\mu\{r\rightarrow s\}$, 
which we denote $U_\mu$. 
The specialization map $\rho_k$ induces a map, denoted by the same symbol, 
\[\rho_k:\mathrm{MS}_{\Gamma_0(M)}^{\ord,\dagger}(\mathbb{D}_*)\longrightarrow
\mathrm{MS}_{\Gamma_0(N)}(V_{k-2}(\C_p)), \]
and \cite[Theorem 1.5]{BD-HidaFamilies} shows that for each choice of sign $\pm$ 
there exists a neighborhood $U$ of $2$ in $\mathcal{X}$ and $\mu_*^\pm\in
 \mathrm{MS}_{\Gamma_0(M)}^{\ord,\dagger}(\mathbb{D}_*)$ 
 such that $\rho_2(\mu_*^\pm)=I_f^\pm$ 
and for all integers $k\in U$, $k\geq 2$, there is $\lambda^\pm(k)\in\C_p$
such that $\rho_k(\mu_*^\pm)=\lambda^\pm(k)I_{f_k}^\pm$; also, $U$ can be 
chosen so that $\lambda^\pm(k) \neq 0$ for all $k\in U$. 

\begin{theorem}\label{logtheorem}
If $Q\in \mathcal F_{Dc^2}$, then 
\[\log_q(P_{\tau_Q})=\int_{(\Z_p^2)'}\log_q(x-\tau_Q y) \, d\mu_*^\pm\{r\rightarrow \gamma_\tau(r)\}(x,y)\]
\end{theorem}
\begin{proof} This follows from \cite[Theorem 2.5]{BD-Rationality} as in 
\cite[Corollary 2.6]{BD-Rationality} noticing that 
the set \[\{(x,y)\in \Q_p^2\ | \ x-\tau_Q y\in \mathcal{O}_K\otimes\Z_p\}\] coincides with $\Z_p^2$. 
\end{proof}

\section{Complex $L$-functions of real quadratic fields}\label{Sec:SV} 

Here we recast the special value formula of the second author and 
Whitehouse \cite{MW}, restricted to the setting of this paper, 
in a form convenient for our purposes. 

Let $f\in S_k(\Gamma_0(M))$ be a even weight $k\geq2$ newform for $\Gamma_0(M)$. 
Let $F/\Q$ be a real quadratic field of discriminant $D>0$, prime to $M$, and let 
$\chi_D$ be the associated quadratic Dirichlet character;  
with a slight abuse of notation, we will 
denote by the same symbol $\chi_D:\mathbb{A}_\Q^\times\rightarrow \C^\times$ 
the associated Hecke character, where $\mathbb A_\Q$ is the adele ring of $\Q$. We assume that 
all primes $\ell\mid M$ are split in $F$. 

Let $c$ be an integer prime to $DM$
and let $H_c^+$ be the strict ring class field of $F$ of conductor $c$.
Let $G_c^+=\Gal(H_c^+/F)$. 
Let $\chi:G_c^+\rightarrow \C^\times$ be a primitive character, namely, a character which does not 
factor through $G_f^+$ for any proper divisor $f\mid c$;
with a slight abuse of notation, we will 
denote by the same symbol $\chi:\mathbb{A}_F^\times\rightarrow \C^\times$ 
the associated Hecke character, where $\mathbb A_F$ is the adele ring of $F$. 

\subsection{Optimal embedding theory}\label{sec:optemb}
We set up the theory of optimal embeddings, and its relation to the strict, or narrow, class group of 
$\mathcal{O}_c$ and quadratic forms of discriminant $Dc^2$. For more details, 
see \cite[\S4.3]{LV} and \cite[\S4.1]{LRV}; for the general theory of optimal embeddings, see 
\cite{Vigneras} and \cite[\S\S3,4]{LV-MM}.  

Let us denote by $\mathcal B=\M_2(\Q)$ the split quaternion algebra over $\Q$ 
and denote by $R_0$ the order in $\mathcal B$ consisting of matrices in $\M_2(\Z)$ 
which are upper triangular modulo $M$.  
Let $\mathcal O_c=\Z+c\cdot\mathcal O_F$ be the order of $F$ of 
conductor $c$, where $\mathcal{O}_F$ is the ring of integers of $F$.
Let $\mathrm{Emb}(\mathcal{O}_c,R_0)$ be the set of optimal embeddings
$\psi:F\rightarrow\mathcal{B}$ of $\mathcal{O}_c$ into $R_0$ (so $\psi(\mathcal{O}_c)=R_0\cap\psi(F)$). 
For every prime $\ell\mid M$ fix orientations of $R_0$ and $\cO_c$ at $\ell$, i.e., ring homomorphisms 
$\mathfrak O_\ell:R_0\rightarrow\mathbb F_{\ell}$ and 
$\mathfrak o_\ell:\mathcal O_c\rightarrow\mathbb F_{\ell}$. 
Two embeddings $\psi,\psi'\in\Emb\bigl(\cO_c,R_0\bigr)$ are said to have
{the same orientation} at a prime $\ell \mid M$ if 
$\mathfrak O_\ell\circ(\psi|_{\cO_c})=\mathfrak O_\ell\circ(\psi'|_{\cO_c})$ and are said to have {opposite orientations} at $\ell$ otherwise. 
An embedding $\psi\in\Emb\bigl(\cO_c,R_0\bigr)$ is said to be \emph{oriented} if $\mathfrak O_\ell\circ(\psi|_{\cO_c})=\mathfrak o_\ell$ for all primes $\ell \mid M$.
We denote the set of oriented optimal embeddings of $\cO_c$ into $R_0$ by 
$\E(\cO_c,R_0)$.
The action of $\Gamma_0(M)$ on $\Emb(\cO_c,R_0)$ from the right 
by conjugation restricts to an action on 
$\E(\cO_c,R_0)$. If $\psi\in\E(\cO_c,R_0)$ then
$\psi^*:=\omega_\infty\psi \omega_\infty^{-1}$ 
belongs to $\E(\cO_c,R_0)$ as well, where recall that $\omega_\infty=\smallmat 100{-1}$, 
and $\psi$ and $\psi^*$ have 
opposite orientations at all $\ell\mid M$.
If $\ell$ is a prime dividing $M$ then $\psi$ and $\omega_\ell\psi\omega_\ell^{-1}$,
where $\omega_\ell=\smallmat 0{-1}\ell0$, 
have opposite orientations at $\ell$ and the same orientation at all primes dividing $M/\ell$.

Let $\mathfrak{a}\subset\cO_c$ be an ideal representing a class $[\mathfrak{a}]\in\Pic^+(\cO_c)$ and let $\psi\in\Emb(\cO_c,R_0)$. The left $R_0$-ideal $R_0\psi(\mathfrak{a})$ is principal; let $a\in R_0$ be a generator of this ideal with positive reduced norm, which is unique up to elements in $\Gamma_0(M)$. 
The right action of $\psi(\cO_c)$ on $R_0\psi(\mathfrak{a})$ shows that $\psi(\cO_c)$ is contained in the right order of $R_0\psi(\mathfrak{a})$, which is equal to $a^{-1}R_0a$. This 
defines an action of $\Pic^+(\cO_c)$ on conjugacy classes of embeddings given by
$[\mathfrak{a}]\cdot[\psi]=\bigl[a\psi a^{-1}\bigr]$ in $\Emb(\cO_c,R_0)/\Gamma_0(M)$.
The principal ideal $(\sqrt{D})$ is a proper $\cO_c$-ideal of $F$; 
denote $\mathfrak D$ its class in $\Pic^+(\cO_c)$ and define 
$\sigma_F:=\text{rec}(\mathfrak D)\in G_c^+$, where $\mathrm{rec}$ 
is the arithmetically normalized reciprocity map of class field theory. If $\mathfrak{a}=(\sqrt{D})$ then we can take
$a=\omega_\infty\cdot\psi(\sqrt{D})$ in the above discussion, 
which shows that  
$\mathfrak{D}\cdot[\psi]=\bigl[\omega_\infty\psi\omega_\infty^{-1}\bigr]=[\psi^\ast]$.
Using the reciprocity map of class field theory, 
for all $\sigma\in G_c^+$ and $[\psi]\in\Emb(\cO_c,R_0)/\Gamma_0(M)$ define
$\sigma\cdot[\psi]:=\text{rec}^{-1}(\sigma)\cdot[\psi]$. 
In particular, $\sigma_F\cdot[\psi]=[\psi^*]$ 
for all $\psi\in\Emb(\cO_c,R_0)$. 

If $\psi$ is an oriented optimal embedding
then the Eichler order $a^{-1}R_0a$ inherits an orientation from the one of $R_0$ and it can be checked that we get an induced action of $\Pic^+(\cO_c)$ (and $G_c^+$) on the set 
$\E(\cO_c,R_0)/\Gamma_0(M)$, and this action is free and transitive. To describe a (non-canonical) bijection between $\E(\cO_c,R_0)/\Gamma_0(M)$ and $G_c^+$, 
fix once and for all an auxiliary embedding $\psi_0\in\E(\cO_c,R_0)$; 
then 
$\sigma\mapsto\sigma\cdot[\psi_0]$ defines a bijection 
$E:G_c^+\rightarrow\E(\cO_c,R_0)/\Gamma_0(M)$ 
whose inverse, $G=E^{-1}:\E(\cO_c,R_0)/\Gamma_0(M)\longrightarrow G_c^+$ satisfies 
the relation 
$G([\psi^*])=\sigma_F\cdot G([\psi])$
for all $\psi\in\E(\cO_c,R_0)$.    
Choose for every $\sigma\in G_c^+$ an embedding
$\psi_\sigma\in E(\sigma)$, 
so that the family $\{\psi_\sigma\}_{\sigma\in G_c^+}$ is a set of representatives of the $\Gamma_0(M)$-conjugacy classes of oriented optimal embeddings of $\cO_c$ into $R_0$. If $\gamma,\gamma'\in R_0$ write $\gamma\sim\gamma'$ to indicate that $\gamma$ and $\gamma'$ are in the same $\Gamma_0(M)$-conjugacy class, and adopt a similar notation for (oriented) optimal embeddings of $\cO_c$ into $R_0$. For all $\sigma,\sigma'\in G_c^+$ one has $\sigma\cdot\psi_{\sigma'}\sim\psi_{\sigma\sigma'}$ and 
$\psi_\sigma^*\sim \psi_{\sigma_F\sigma}$ 
for all $\sigma\in G_c^+$. 

Finally, note that the set $\mathcal{E}(\mathcal{O}_c,R_0)/\Gamma_0(M)$ is in bijection with $\mathcal{F}_{Dc^2}/\Gamma_0(M)$, since 
both sets are in bijection with $G_c^+$; explicitly, to the class of the oriented optimal embedding $\psi$ corresponds the class of the quadratic form \[Q_\psi(x,y)=Cx^2-2Axy-By^2\] with 
$\psi(\sqrt{D}c)=\smallmat ABC{-A}$.

\subsection{Adelic ring class groups}
Below we will want to view the ring class group $G_c^+$ adelically.
Since this is omitted from the literature on class field 
theory that we are aware of (adelic treatments usually explain ray class fields but 
not ring class fields, and expositions of ring class groups which treat real 
quadratic extensions, e.g., \cite{cohn}, tend to not use adelic language), we explain
briefly the passage from classical ring class groups to adelic ring class groups
here.  For a point of reference, we also describe the relation with ray class groups.
As it causes little extra difficulty, in this subsection only,
we allow $F$ to be an arbitrary (real or imaginary) quadratic field of discriminant $D$ and do not require $c$ to be coprime to $D$.

Let $c \in \N$ and $\mathfrak m = \mathfrak m_f \mathfrak m_\infty$ where
$\mathfrak m_f = c \mathcal O_F$ and $\mathfrak m_\infty$ is a subset of the real
places of $F$.  For a real place $v$, let $\sigma_v$ be the associated embedding  
of $F$ into $\R$. Let $J_{\mathfrak m}$ be the group of fractional ideals of
$\mathcal O_F$ which are prime to $\mathfrak m_f$.  Let $F_{\mathfrak m}^1$
be the subset of $F^\times$ consisting of $x \in F^\times$ such that $\sigma_v(x) > 0$
for each $v \in \mathfrak m_\infty$ and $v_{\mathfrak p}(x-1) \ge v_{\mathfrak p}(c)$ for $\mathfrak p | \mathfrak m_f$.  Let $P_{\mathfrak m}^1$ denote the
set of principal ideals generated by elements of $F_{\mathfrak m}^1$.
Then the ray class group mod $\mathfrak m$ of $F$ is 
$\mathrm{Cl}_{\mathfrak m}(F) = J_{\mathfrak m}/P^1_{\mathfrak m}$.

Let $F_{\mathfrak m}^\Z$ be the set of $x \in F^\times$ such that
$\sigma_v(x) > 0$ for each $v \in \mathfrak m_\infty$ and for each
$\mathfrak p | \mathfrak m_f$ there exists $a \in \Z$ coprime to $c$ such that
$v_{\mathfrak p}(x-a) \ge v_{\mathfrak p}(c)$.  Let $P^{\Z}_{\mathfrak m}$
be the set of principal ideals in $F$ generated by elements of 
$F_{\mathfrak m}^\Z$.  Then the ring class group mod $\mathfrak m$ of
$F$ is $G_{\mathfrak m}(F) = J_{\mathfrak m}/P^\Z_{\mathfrak m}$.
Note we can write $F_{\mathfrak m}^\Z = \bigcup_{a \in (\Z/c\Z)^\times} 
a F_{\mathfrak m}^1$.  Hence $\mathrm{Cl}_{\mathfrak m}(F) / \mathrm{Im} (\Z/c \Z)^\times \simeq G_{\mathfrak m}(F) $, where $\mathrm{Im} (\Z/c \Z)^\times$
denotes the image of the natural map from
$(\Z/c\Z)^\times$ to $\mathrm{Cl}_{\mathfrak m}(F)$,
which is not in general injective.

Via the usual correspondence between ideals and ideles, $J_{\mathfrak m}$
is identified with $\hat F^\times_{\mathfrak m}/ \hat {\mathcal O}_F^\times$,
where $\hat F^\times_{\mathfrak m}$ consists of finite ideles $(\alpha_v)$ such that
$\alpha_v \in \mathcal O_{F,v}^\times$ for all $v | \mathfrak m_f$
and $\hat {\mathcal O}_F^\times = \prod_{v < \infty}  \mathcal O_{F,v}^\times$.  
For $v < \infty$, we put $W_v = \mathcal O_{F,v}^\times$
unless $v \mid \mathfrak m_f$, in which case $W_v = 1+ \mathfrak m_f \mathcal O_{F,v}$.
For $v \mid \infty$, we put $W_v = F_v^\times$ unless 
$v \mid \mathfrak m_\infty$, in which case $W_v = \R_{> 0}$.  
Now define $W = \prod W_v$ and  $\A_{F, \mathfrak m}^1 = \prod'_{v \nmid \mathfrak m} F_v^\times \times \prod_{v | \mathfrak m} W_v$.
Then we have $F_{\mathfrak m}^1 = F^\times \cap \A_{F,\mathfrak m}^1$
and $J_{\mathfrak m} \simeq \A_{F, \mathfrak m}^1 / W$, 
so  $\mathrm{Cl}_{\mathfrak m}(F) = F_{\mathfrak m}^1 \bs \A_{F, \mathfrak m}^1 / W
= F^\times \bs \A_F^\times / W$.

For the ring class group, again we can realize it as a quotient of the idele class group
$F^\times \bs \A_F^\times$, but now it will be a quotient by a subgroup $U = \prod U_\ell \times
U_\infty$ which is a product over rational primes, rather than primes of $F$.
As usual, for a rational prime $\ell < \infty$, write 
$\mathcal O_{F,\ell}$ for $\mathcal O_F \otimes_\Z \Z_\ell$, which is isomorphic to
$\Z_\ell \oplus \Z_\ell$ if $\ell$ splits in $F$ and otherwise is $\mathcal O_{F,v}$ if $v$ is the
unique prime of $F$ above $\ell$.
Now set $U_\ell = \mathcal O_{F,\ell}^\times$ if $\ell \nmid c$ and
$U_\ell = (\Z_\ell + c \mathcal O_{F,\ell})^\times$ if $\ell \mid c$.  We can uniformly write
$U_\ell = \mathcal O_{c,\ell}^\times$ for $\ell < \infty$, where 
$\mathcal O_c = \Z + c \mathcal O_F$ and $\mathcal O_{c,\ell} = \mathcal O_c \otimes_\Z \Z_\ell$. For later use, we will write 
$\hat {\mathcal O}_c^\times  = \prod U_\ell$.  Note that this is different from 
the product $\prod_{v < \infty}  {\mathcal O}^\times_{c,v}$ for $v$ running over primes of 
$F$ if $c$ is divisible by primes which split in $F$.
Put $U_\infty = W_\infty = \prod_{v | \infty} W_\infty$ and
$\A_{F, \mathfrak m}^\Z = \prod'_{v \nmid \mathfrak m} F_v^\times \times \prod_{v | \mathfrak m} U_v$.
Then $F_{\mathfrak m}^\Z = \A_{F, \mathfrak m}^\Z \cap F^\times$ 
and we see the ring class group is
\[ G_{\mathfrak m}(F) = F_{\mathfrak m}^\Z \bs \A_{F, \mathfrak m}^\Z / U
= F^\times \bs \A_F^\times / U. \] 

In our case of interest, namely $F$ is real quadratic
and $\mathfrak m_\infty$ contains both real places of $F$, we write $U_\infty = F_\infty^+$.
Thus we can write our strict ring class group as
\begin{equation} \label{id-ring-class}
 G_c^+ = F^\times \bs \A_F^\times / \hat {\mathcal O}_c^\times F_\infty^+.
\end{equation}

\subsection{Special value formulas} \label{special-values}
We return to our case of interest where $F/\Q$ is real quadratic of discriminant $D$, 
$f$ is a weight $k$ newform for $\Gamma_0(M)$,
$c$ is an integer coprime with $D M$, and $\mathcal{O}_c=\Z+c\mathcal{O}_F$. 
Let $H_c$ be the corresponding ring class field and $h_c$ be the degree of 
$H_c/F$, which coincides 
with the cardinality of $\Pic(\mathcal{O}_c)$. Denote by $h_c^+$ the cardinality of $G_c^+$, 
so $h_c^+/h_c$ is equal to $1$ or $2$. Fix ideals $\mathfrak{a}_\sigma$ 
for all $\sigma\in G_c=\Gal(H_c/F)$ in such a way that $\Sigma_c=\{\mathfrak{a}_\sigma\ |\ \sigma\in G_c\}$ is a complete system of representatives for $\Pic(\mathcal{O}_c)$. 
Clearly $\Sigma_c^+=\Sigma_c$ is also a system of representatives for $\Pic^+(\mathcal{O}_c)$ 
if $h_c^+=h_c$, while if $h_c^+\neq h_c$ the set $\Sigma_c^+$ of representatives 
of $\Pic^+(\mathcal{O}_c)$ can be written as $\Sigma_c\cup\Sigma_c'$ with 
$\Sigma_c'=\{\mathfrak{d}\mathfrak{a}_\sigma\ |\ \sigma\in G_c\}$ and $\mathfrak{d}=(\sqrt{D})$. 
Let $\epsilon_c > 1$ be the smallest totally positive power of a fundamental unit in 
$\mathcal{O}_c^\times$,  and 
for all $\sigma\in G_c^+$ define $\gamma_\sigma=\psi_\sigma(\epsilon_c)$.
Finally, define 
\begin{equation}\label{alpha}\alpha= \prod_{\ell\mid c, \, {D \leg \ell} = -1}
\ell,
\end{equation}
where $\ell$ runs over all rational primes dividing $c$ which are inert in $F$.

Denote by $\pi_f$ and $\pi_\chi$ the 
automorphic representations of $\GL_2(\A_\Q)$ attached to $f$ and $\chi$,
respectively.

\begin{theorem} \label{thm:sv} Let $c$ be an integer such that $(c,DM)=1$. 
Let $\chi$ be a character of $G_c^+$ such that  
the absolute norm of the conductor of $\chi$ is $c(\chi) = c^2$.
For any choice of the base point $\tau_0\in\mathcal{H}$, we have 
\[L(\pi_{f}\otimes\pi_\chi,1/2)=\frac{4}{\alpha^2 \cdot (Dc^2)^{(k-1)/2}}
\left|\sum_{\sigma\in G_c^+}{\chi^{-1}(\sigma)}\int_{\tau_0}^{\gamma_\sigma(\tau_0)}f(z)Q_{\psi_\sigma}(z,1)^{(k-2)/2} \, dz\right|^2.\]
\end{theorem} 

When $c=1$, this is the positive weight case of \cite[Theorem 6.3.1]{Popa}, which also treats
weight 0 Maass forms.  If desired, one could similarly extend the above result to
 weight 0 Maass forms.

\begin{proof} 
Write $\pi := \pi_f = \otimes'_v \pi_v$, where $v$ runs over all places of $\Q$, 
and let $n_\ell(\pi)$ be the conductor of $\pi_\ell$ for each prime number $\ell$.
Define 
\[ U_f(M) = \prod_\ell U_\ell(M), \quad U_\ell(M) = \left\{ \mat abcd \in \GL_2(\Z_\ell) : c \equiv 0 \mod M \right\}. \]
We associate to $f$ the automorphic form $\varphi_\pi=\varphi_f$ on $\GL_2(\A_\Q)$ given by
\[ \varphi_\pi : Z(\A) \GL_2(\Q) \bs \GL_2(\A_\Q) / U_f(M) \longrightarrow \C \]
\[ \varphi_\pi \bmx a & b \\ c & d \emx = 2 j(g;i)^kf \left(\frac{ a i + b}{ ci + d} \right),  \]
for $g=\smallmat abcd \in \GL_2(\R)^+$, where we write 
$j(g; i)=\det(g)^{1/2} (ci+d)^{-1}$ for the automorphy factor.
Then $\varphi_\pi$ is $R_{0,\ell}$-invariant for each finite prime $\ell$.
The scaling factor of 2 is present so that the archimedean zeta integral of
$\varphi_\pi$ gives the archimedean $L$-factor.

For $\phi\in \pi$, let 
\[(\phi,\phi)=\int_{Z(\mathbb{A}_\Q)\GL_2(\Q)\backslash \GL_2(\mathbb{A}_\Q)}\phi(t)\overline{\phi(t)}\, dt\] 
be the Petersson norm of $\phi$, where we take as measures on the groups 
$\GL_2(\mathbb{A}_\Q)$ and $Z(\A_\Q)$ the products of the 
local Tamagawa measures.  Here, as usual, we take the quotient measure on
the quotient, giving $\GL_2(\Q) \subset \GL_2(\A_\Q)$ the counting measure.

Denote by $\pi_F$  the base change of $\pi$ to $F$.  Let  
$L(\pi_F\otimes\chi,s)$ be the $L$-function of $\pi_F$ twisted by $\chi$, which
equals the Rankin--Selberg $L$-function $L(\pi_{f}\otimes\pi_\chi,s)$.
Since $F$ splits at each prime $\ell$ where $\pi$ is ramified and at $\infty$,
$\epsilon(\pi_{F,v} \otimes, \chi_v, 1/2) = +1$ for all places $v$ of $\Q$.
Then, in our setting, the main result in \cite[Theorem 4.2]{MW} states that
\begin{equation} \label{eq:mw1}
\frac{\left|P_\chi(\varphi)\right|^2}{(\varphi,\varphi)}=
\frac{L(\pi_F\otimes\chi,1/2)}{(\varphi_\pi',\varphi_\pi')} \cdot
\frac{4}{\sqrt{D c(\chi)}}\cdot \prod_{\ell\mid c}\left(\frac{\ell}{\ell-\chi_D(\ell)}\right)^2,
\end{equation}
where $\varphi \in \pi$ is a suitable test vector,
\[P_\chi(\varphi)=\int_{F^\times \mathbb{A}_{\Q}^\times\backslash \mathbb{A}_F^\times}\varphi(t) \chi^{-1}(t) \, dt, \]
and $\varphi_\pi'$ (denoted $\varphi_\pi$ in \emph{loc.\ cit.}) is a vector in $\pi$
differing from $\varphi_\pi$ only at $\infty$.  
We describe $\varphi$ and $\varphi_\pi'$ precisely below.
Similar to before, we take the products of local Tamagawa measures on $\A_F^\times$ and $\A_\Q^\times$,
and give $F^\times$ the counting measure.

First we describe the choice of the test vector $\varphi$, which we only need to
specify up to scalars, as the left-hand side above is invariant under scalar 
multiplication. We will take 
$\varphi=\otimes'_v\varphi_v$, 
where $v$ runs over all places of $\Q$.  
For $\ell$ a finite prime of $\Q$, let $c(\chi_\ell)$ denote the smallest $n$ such that
$\chi_\ell$ is trivial on $(\Z_\ell + \ell^n \mathcal O_{F,\ell})^\times$.  Since $\chi$ is a character
of $G_c^+$, we have $c(\chi_\ell) \le v_\ell(c)$ for all $\ell$.
In particular, $\chi_\ell$ is trivial on $\Z_\ell^\times$, so $c(\chi_\ell)$ is the smallest $n$ such that
$\chi_\ell$ is trivial on $(1 + \ell^n \mathcal O_{F,\ell})^\times$, and thus agrees with the
usual definition of the conductor of $\chi_\ell$ when $\ell$ is inert in $F$.  Similarly, if $\ell$
is ramified in $F$, say $\ell = {\mathfrak p}^2$, then $c(\chi_\ell)$ is twice the conductor of
$\chi_\ell = \chi_{\mathfrak p}: F_{\mathfrak p}^\times \to \C^\times$, though this case does not
occur by our assumption $(c,D) = 1$.
If $\ell = \mathfrak p_1 \mathfrak p_2$ is split in $F$, then we can write 
$\chi_\ell = \chi_{\mathfrak p_1} \otimes \chi_{\mathfrak p_2}$ with $\chi_{\mathfrak p_1}$,  
$\chi_{\mathfrak p_2}$ characters of $\Q_\ell^\times$, which
are inverses of each other on $\Z_\ell^\times$ as $\chi_\ell$ is trivial on $\Z_\ell^\times$.  
Hence $\chi_{\mathfrak p_1}$ and $\chi_{\mathfrak p_2}$ have the same
conductor, which is $c(\chi_\ell)$.  Consequently, $c(\chi)$, the absolute norm of the 
conductor of  $\chi$, is 
\begin{equation} \label{c-chi}
c(\chi) = \mathrm{Norm}\left(\prod_{\ell} \ell^{c(\chi_\ell)}\right)= \prod_{\ell} \ell^{2c(\chi_\ell)}.
\end{equation}

Note that since $(c,M) = 1$, we have $c(\chi_\ell) = 0$ whenever $\pi_\ell$ is ramified, i.e.,
the conductor $c(\pi_\ell) > 0$.  
If $c(\chi_\ell) = 0$, 
let $R_{\chi, \ell}$  be an Eichler order of reduced discriminant $\ell^{c(\pi_\ell)}$ in 
$M_2(\Q_\ell)$ containing $\mathcal O_{F,\ell}$.  If $c(\chi_\ell) > 0$, so $\pi_\ell$ is 
unramified, let $R_{\chi,\ell}$ be a maximal order of $M_2(\Q_\ell)$ which optimally contains
$\Z_\ell + \ell^{c(\chi_\ell)} \mathcal O_{F,\ell}$.   In either case, $R_{\chi,\ell}$ is unique up to conjugacy and pointwise
fixes a 1-dimensional subspace of $\pi_\ell$.  For $\ell < \infty$, take $\varphi_\ell \in
\pi_\ell^{R_{\chi, \ell}}$ nonzero, normalized in such a way that $\otimes' \varphi_\ell$ converges.
For instance, we can take $\varphi_\ell = \varphi_{\pi,\ell}$ at almost all $\ell$.
Each $\varphi_\ell$ is a local Gross--Prasad test vector, and our assumptions
imply that the local Gross--Prasad test vectors $\varphi_\ell$ are (up to scalars) translates
of the new vectors $\varphi_{\pi,\ell}$.  (Gross--Prasad test vectors are not
translates of new vectors in general.)

Embed $F$ into $\M_2(\Q)$ as follows.
Consider a quadratic form
\[ Q(x,y) = - \frac C2 x^2 + Axy + \frac B2 y^2 \in \mathcal F_{Dc^2}. \]
This means $Q$ is primitive of discriminant $Dc^2=A^2+BC$,
$2 \mid B$ and $2M \mid C$, which implies $A^2 \equiv Dc^2 \mod 4$.
Take the embedding of $F$ into $\M_2(\Q)$ induced by
$\sqrt Dc \mapsto \smallmat ABC{-A}$. 
Then $\mathcal O_c = R_0 \cap F$, and 
\[ F_\infty^\times = \left\{ g(x,y) := \bmx x + Ay & By \\ Cy & x-Ay \emx \in \GL_2(\R)  \right\}. \]
For a prime $\ell \nmid c$, we have ${\mathcal O}_{c,\ell} = {\mathcal O}_{F,\ell}
\subset R_{0, \ell}$.  Thus we may take 
$R_{\chi,\ell} = R_{0,\ell}$ for $\ell \nmid c$ such that $\chi_\ell$ is unramified---in 
particular, for $\ell \nmid cD$.  When $\chi_\ell$ is ramified, we may take
$R_{\chi,\ell} = R_{0,\ell}$ if and only if $c(\chi_\ell) = v_\ell(c)$.
By assumption, $c(\chi) = \prod_{\ell} \ell^{2c(\chi_\ell)} = c^2$, so we may take
$R_{\chi,\ell} = R_{0,\ell}$ at each finite place $\ell$.  Thus we may and will take
$\varphi_\ell$ to be $\varphi_{\pi, \ell}$ at each $\ell$.

Now we describe $\varphi_\infty$.
Note we can identify $F_\infty^\times/\Q_\infty^\times$ with $F_\infty^1/ \{ \pm 1 \}$,
where 
\[ F_\infty^1 = F_\infty^{1, +} \cup F_\infty^{1, -}, \quad
F_\infty^{1, \pm} = \{ g(x,y)  \in F_\infty^\times : \det g(x,y) = x^2-Dc^2y^2 = \pm 1 \}. \]
Let
\[ \gamma_\infty = \bmx A+\sqrt{D}c & A-\sqrt{D}c \\ C & C \emx. \]
Then
\[ \gamma_\infty^{-1} \bmx A & B \\ C & -A \emx \gamma_\infty = 
\bmx \sqrt Dc & 0\\0 & - \sqrt Dc \emx. \]
So
\[ \gamma_\infty^{-1} F^{1}_\infty \gamma_\infty = \left\{ \bmx x + y \sqrt Dc &0 \\0 & x-y \sqrt Dc \emx : x^2-Dc^2y^2 = \pm 1 \right\}. \]

 The maximal compact subgroup of $F^1_\infty$ is 
 \[ \Gamma_F = \gamma_\infty \left\{ \bmx \pm 1 &0 \\ 0&  \pm 1 \emx \right\} \gamma_\infty^{-1} 
 = \{ \pm I, \pm g(0, -(\sqrt D c)^{-1}) \}, \]
 where one reads the $\pm$ signs independently.
Let $U_\infty = \gamma_\infty \mathrm{O}(2) \gamma_\infty^{-1}$,
where $\mathrm{O}(2)$ denotes the standard maximal compact subgroup of 
$\GL_2(\R)$.  Then $U_\infty \supset \Gamma_F$, and the archimedean test vector in 
\cite{MW} is the unique up to scalars nonzero vector $\varphi_\infty$ lying
in the minimal $U_\infty$-type such that $\Gamma_F$
acts by $\chi_\infty$ on $\varphi_\infty$.
Specifically, we can take 
\begin{equation} \label{eq:phi-infty-sum}
\varphi_\infty = \pi_\infty(\gamma_\infty) (\varphi_{\infty,k} \pm \varphi_{\infty,-k}),
\end{equation} 
where $\varphi_{\infty,\pm k} = \frac 12 \pi_\infty \smallmat {\pm1}001\varphi_\pi$ 
is a vector of weight $\pm k$ in $\pi_\infty$, and the $\pm$ sign in \eqref{eq:phi-infty-sum}
matches the sign of $\chi_\infty \smallmat{-1}001$.
This completely describes the test vector $\varphi$ chosen in \cite{MW}.

For our purposes, we would like to work with a different archimedean component
than $\varphi_\infty$, corresponding to (a translate of) $\varphi_\pi$.  Let $\varphi^-$ be the pure tensor in $\pi$ which
agrees with $\varphi$ at all finite places, and is defined like $\varphi_\infty$ at
infinity
except using the opposite sign in the sum \eqref{eq:phi-infty-sum}.  
Then necessarily any 
$\chi_\infty$-equivariant linear function on $\pi_\infty$ kills $\varphi^-_\infty$, so
$P_\chi(\varphi^-) = 0$.  Hence $P_\chi(\varphi) = P_\chi(\varphi')$ where
$\varphi' = \varphi + \varphi^-$, and we can write $\varphi' = \otimes \varphi_v'$,
where $\varphi_\ell' = \varphi_\ell$ for finite primes $\ell$ and 
$\varphi'_\infty = \pi_\infty(\gamma_\infty)\varphi_\pi$, i.e.,   $\varphi'(x)=\varphi_\pi(x \gamma_\infty)$.

Finally, we describe the vector $\varphi_\pi'$ appearing in \eqref{eq:mw1}.  
It is defined to be a factorizable function in $\pi$ whose associated local Whittaker functions
are new vectors whose zeta functions are the local $L$-factors of $\pi$ at finite places,
and at infinity is the vector in the minimal ${\mathrm O}(2)$-type that transforms by $\chi_\infty$
under $\smallmat{\pm 1}{0}{0}{\pm 1}$ such that the associated Whittaker
function (restricted to first diagonal component) at infinity is 
$W_\infty(t) = 2\chi_\infty\smallmat{t}{0}{0}{1}|t|^{k/2} e^{-2\pi |t|}$.  
(This normalization gives $L_\infty(s,\pi) = \int_0^\infty W_\infty(t) |t|^{s-1/2} \, d^\times t$.)  
Thus $\varphi_{\pi}'$ agrees with $\varphi_\pi$ at all finite places
and $\varphi_{\pi,\infty}' = 2(\varphi_{\infty,k} \pm \varphi_{\infty,-k})$,
where the $\pm$ sign matches that in \eqref{eq:phi-infty-sum}.

Hence $\varphi = \frac 12 \pi(\gamma_\infty) \varphi'_\pi$, so by invariance of the inner product
we have $(\varphi, \varphi) = \frac 14 (\varphi'_\pi, \varphi'_\pi)$,
 and \eqref{eq:mw1} becomes 
\begin{equation} \label{eq:mw2}
\left| P_\chi(\varphi')\right|^2= \left| P_\chi(\varphi)\right|^2 =
L(\pi_F\otimes\chi,1/2)\cdot
\frac{1}{\sqrt{D}c}\cdot \prod_{\ell\mid c}\left(\frac{\ell}{\ell-\chi_D(\ell)}\right)^2.
\end{equation}

Now we want to rewrite $P_\chi(\varphi')$.
Recall that $\epsilon_c > 1$ is the smallest totally positive power of a fundamental unit in 
$\mathcal{O}_c^\times$. 
From \eqref{id-ring-class}, we obtain the isomorphism
\[  \A_\Q^\times F^\times \bs \A_F^\times / \hat{\mathcal{O}}_c^\times  
\simeq G^+_c \cdot(F^{+}_\infty / \langle \epsilon_c \rangle \Q_\infty^+)
 \simeq G^+_c \cdot(F^{1,+}_\infty / \langle \pm \epsilon_c \rangle). \]
We may identify
\[ F^{1,+}_\infty/ \langle \pm \epsilon_c \rangle = 
\left\{ \bmx x + Ay & By \\ Cy & x-Ay \emx \in \SL_2(\R) :  1 \le x + y \sqrt Dc < \epsilon_c \right\}, \]
and the orbit of $\gamma_\infty i$ in the upper half plane by this set is the geodesic
segment connecting $\gamma_\infty i$ to $\epsilon_c \gamma_\infty i$,
i.e., the image under $\gamma_\infty$ of $\{ i y : 1 \le y \le \epsilon_c^2 \} \subset
\mathcal H$.
Let us call this arc $\Upsilon$.

Since $\A_\Q^\times \subset F^\times \hat {\mathcal O}_c^\times F_\infty^+$ 
and $G_c^+ \simeq F^\times  \bs \A_F^\times/ \hat{\mathcal O}_c^\times F_\infty^+$ 
where $F_\infty^+ = (\R_{> 0})^2$, we see that $\chi$ is trivial on $\A_\Q^\times \hat {\mathcal O}_c^\times F_\infty^+$.  The Tamagawa measure gives
$\vol(F^\times \A_\Q^\times \bs \A_F^\times) = 2 L(1, \eta) = 4 h_F \log \epsilon_F 
\vol(\hat{\mathcal O}^\times)$, where $\eta$ is the quadratic character of 
$\A_\Q^\times / \Q^\times$  attached to $F/\Q$ and $\epsilon_F$ is the fundamental unit 
of $F$.  This implies $\vol(\A_\Q^\times F^\times \bs \A_F^\times / \hat{\mathcal{O}}_c^\times) = 
2 h_c^+ \mathrm{len}(\Upsilon)$, where $\mathrm{len}(\Upsilon) = 2 \log \epsilon_c$
is the length of $\Upsilon$ with respect to the usual hyperbolic distance.
Thus we compute
\begin{align*}
 P_\chi(\varphi') &= 2\vol(\hat {\mathcal{O}}_c^\times) \sum_{t \in G_c^+} \chi^{-1}(t)
\int_{F^{1,+}_\infty/ \langle \pm \epsilon_c \rangle} \varphi_\pi(tg \gamma_\infty) \, dg \\
&= 4 \vol(\hat {\mathcal{O}}_c^\times) \sum_{t \in G_c^+} \chi^{-1}(t)
\int_1^{\epsilon_c} j\left(t \gamma_\infty \smallmat u00{u^{-1}};i\right)^{k} 
f\left(t\gamma_\infty\smallmat u00{u^{-1}} \cdot i\right) \, d^\times u  \\
&= 2 \vol(\hat {\mathcal{O}}_c^\times) \sum_{t \in G_c^+} \chi^{-1}(t)
\int_1^{\epsilon_c^2} j\left(t \gamma_\infty \smallmat y001; i\right)^{k} 
f\left(t\gamma_\infty\smallmat y001 \cdot i\right) \, d^\times y,
\end{align*}
where we use that $f$ has trivial central character and the substitution $y=u^2$ at the last step.

For $\ell$ a rational prime dividing $c$, note that
${\mathcal O}_{F,\ell}^\times/ {\mathcal O}_{c,\ell}^\times \simeq
\Z_\ell^\times/(1+c\Z_\ell)$ when $\ell$ splits in $F$ and
${\mathcal O}_{F,\ell}^\times/ {\mathcal O}_{c,\ell}^\times \simeq
({\mathcal O}_{F,\ell}^\times / (1 + c {\mathcal O}_{F,\ell} )) /
(\Z_\ell^\times/(1+c \Z_\ell))$ when $\ell$ is inert in $F$.
Hence, with our choice of measures,
\[ \vol(\hat {\mathcal{O}}_c^\times) = \vol(\hat{\mathcal O}_F^\times) 
\prod_{\ell | c} [ {\mathcal O}_{F, \ell}^\times : {\mathcal O}_{c, \ell}^\times ]^{-1} = 
\frac 1{\sqrt D}
\prod_{\ell | c, \, {D \leg \ell} = 1} \frac 1{(\ell - 1) \ell^{v_\ell(c) - 1}} \cdot
\prod_{\ell | c, \, {D \leg \ell} = -1} \frac 1{(\ell + 1) \ell^{v_\ell(c)}}, 
\]
where $\ell$ runs over rational primes.

The next task is then to rewrite the integral appearing the in right hand side of the above formula. 
Let $z = \gamma_\infty iy$.  Then
\[ z = \frac AC + \frac{\sqrt Dc}C\left(1 - \frac 2{1+iy}\right). \]
Since
\[ \frac{2iy}{(1+iy)^2} = \frac 2{1+iy} - \frac 12 \left( \frac 2{1+iy} \right)^2
= \frac{BC + 2AC z - C^2z^2}{2Dc^2}, \]
we have
\[ j\left( \gamma_\infty \bmx y &0 \\ 0 &1 \emx; i\right)^{-2} = \frac{C(1+iy)^2}{2\sqrt Dc y} 
= \frac{2 i \sqrt Dc}{-Cz^2 + 2Az + B} = \frac{i\sqrt Dc}{Q(z,1)}, \]
and
\[ dz = \frac{2i y\sqrt Dc}{C (1+iy)^2} \,  d^\times y, \quad \text{i.e.} \quad 
d^\times y = \frac{2  \sqrt Dc}{-Cz^2 + 2Az + B} \,dz  = \frac{\sqrt Dc}{Q(z,1)} \, dz. \]
Making the change of variable $z = \gamma_\infty iy$, 
the above expression can be rewritten as
\[
P_\chi(\varphi')=\frac{2 \vol(\hat {\mathcal{O}}_c^\times)}{i^{k/2}\cdot (\sqrt Dc)^{(k-2)/2}}
\cdot\sum_{t \in G_c^+} \chi^{-1}(t)\int_\Upsilon f(tz)
\cdot Q(z,1)^{(k-2)/2} \, dz.\]
After another change of variable $z'=tz$, the above integral becomes
\begin{eqnarray*}
\int_{t\Upsilon}f(z') \cdot Q(t^{-1}z',1)^{(k-2)/2} \, dz'&=&
\int_{t\Upsilon}f(z')\cdot (Q|t^{-1})(z',1)^{(k-2)/2} \, dz'\\
&=&\int_{\tau_{t}}^{(t\epsilon_c t^{-1})\tau_t}f(z')\cdot (Q|t^{-1})(z',1)^{(k-2)/2} \, dz', 
\end{eqnarray*}
where $\tau_{t}=t{\gamma_\infty i}$. Now, as long as $t$ varies in $G_c^+$, 
the quadratic forms $Q|t^{-1}$ are representatives for the classes in $\mathcal{F}_{Dc^2}/\Gamma_0(M)$, 
as discussed in \S \ref{sec:optemb}.  Moreover, since
$\Upsilon$ is closed in $\mathcal H/\Gamma_0(N)$, this integral does not
depend on the choice of base point.  Plugging this into \eqref{eq:mw2} gives
the asserted formula.
\end{proof}

\subsection{Genus fields attached to orders} \label{sec:34}
Assume from now on that $c$ is odd. 
The \emph{genus field} attached to the order $\mathcal{O}_c$ of discriminant 
$Dc^2$ is the finite abelian extension of $\Q$, with Galois group isomorphic to 
copies of $\Z/2\Z$, contained in the strict class field $H_c^+$ of 
$F$ of conductor $c$ and generated by  
the quadratic extensions 
$\Q(\sqrt{D_i})$ and $\Q(\sqrt{\ell^*})$ where $D=\prod_iD_i$ is any possible 
factorization of $D$ into primary discriminants, 
$\ell\mid c$ is a prime number and $\ell^*=(-1)^{(\ell-1)/2}\ell$. 
See \cite[pp.\ 242-244]{cohn} for details. 

Fix a quadratic character $\chi:G_c^+\rightarrow\{\pm1\}$.

\begin{definition} 
We say that $\chi$ is \emph{primitive} if it does not 
factor through $H_f^+$ for a proper divisor $f\mid c$.
\end{definition} 

We assume that $\chi:G_c^+\rightarrow\{\pm1\}$ is primitive. 
By \eqref{c-chi}, this means $c(\chi) = c^2$.
Then $\chi$ 
cuts out a quadratic extension 
$H_\chi/F$ which, by genus theory for $\mathcal{O}_c$, is biquadratic over $\Q$.
Each quadratic extension of $\Q$ contained in the genus field of the 
order $\mathcal{O}_c$ is of the form $\Q(\sqrt{\Delta})$ for some 
$\Delta=D'\cdot\prod_{j=1}^s\ell_j^*$, with $\ell_j\mid c$ and 
$D'$ a fundamental discriminant dividing $D$. 
Write $H_\chi=\Q(\sqrt{\Delta_1},\sqrt{\Delta_2})$, with 
$\Delta_i=D_i\cdot\prod_{j=1}^{s_i}\ell_{i,j}^*$ for $i=1,2$ as above (so $\ell_{i,j}$ are primes 
dividing $c$), and let 
$K_1=\Q(\sqrt{\Delta_1})$ and $K_2=\Q(\sqrt{\Delta_2})$. 
Since the third quadratic extension contained in $H_\chi$ 
is the quadratic extension is $\Q(\sqrt{D})$,  
we have $\Delta_1\cdot\Delta_2\equiv D \cdot x^2$ for some $x \in \Q^\times$.
We can write $\Delta_1=D_1d$ and $\Delta_2=D_2d$ for some $d=\prod_{j=1}^s\ell_j^*$ with 
$\ell_i\mid c$ primes and $D=D_1\cdot D_2$ a factorization into fundamental discriminants,
allowing $D_1=D$ or $D_2=D$. 
If $d\neq \pm c$, then $\chi$ factors through the extension $H_d^+\neq H_c^+$
by the genus theory of the 
order of conductor $Dd^2$, and therefore $\chi$ is not  a primitive character of $H_c^+$. So $d=\pm c$. 
Thus we conclude that the quadratic fields $K_1=\Q(\sqrt{D_1d})$ 
and $K_2=\Q(\sqrt{D_2d})$
satisfy the following properties: 
\begin{itemize}
\item $D_1\cdot D_2=D$, where $D_1$ and $D_2$ are two coprime fundamental discriminants 
(possibly equal to $1$). 
\item $d=\pm c$ and $d$ is a fundamental discriminant. 
\end{itemize}

Let 
$\chi_{D_1d}$ and $\chi_{D_2d}$ be the quadratic characters attached to the 
extensions $K_1$ and $K_2$ respectively; thus 
$\chi_{D_1d}(x)=\left(\frac{D_1d}{x}\right)$ and $\chi_{D_2d}(x)=\left(\frac{D_2d}{x}\right)$.  
Similarly, let $\chi_D$ be the quadratic character attached to the extension $F/\Q$, i.e., 
$\chi_D(x)=\left(\frac{D}{x}\right)$. In particular, 
for all $\ell\nmid c$ we have 
\begin{equation}\label{characters}
\chi_D(\ell)=\chi_{D_1d}(\ell)\cdot\chi_{D_2d}(\ell).\end{equation}

Say that $\chi$ has sign $+1$ if $H_\chi/F$ is totally real, and sign $-1$ otherwise. 
If $\chi$ has sign $w_\infty\in\{\pm1\}$, put $I_f=I_f^{w_\infty}$ and 
$\Omega_f=\Omega_f^{w_\infty}$.  
Define 
\[\mathbb{L}(f,\chi):=\sum_{\sigma\in G_{{c}}^+}
\chi^{-1}(\sigma)I_f \{ \tau_0\rightarrow\gamma_{\psi_\sigma}(\tau_0) \}
\big(Q_{\psi_\sigma}(x,y)^{(k-2)/2}\big).\]

\begin{lemma}\label{lemmasign}
$\overline{\mathbb{L}(f,\chi)}=w_\infty\cdot\mathbb{L}(f,\chi)$. 
\end{lemma}

\begin{proof}
This follows from the discussion in \cite[\S 6.1]{Popa}. To simplify the notation, define 
\[ \Theta_\psi:= I_f \{ \tau_0\rightarrow\gamma_{\psi}(\tau_0) \}
\big(Q_{\psi}(x,y)^{(k-2)/2}\big),\]
which is independent of the choice of $\tau_0$ and the $\Gamma_0(M)$-conjugacy class of 
$\psi$. Let $z\mapsto\bar{z}$ denote complex conjugation. 
A direct computation shows that 
$\overline{\Theta_\psi}=\Theta_{{\psi}^*}$ where recall that 
${\psi}^*=\omega_\infty\psi \omega_\infty^{-1}$.  
From the discussion in \S\ref{sec:optemb} we have 
$\sigma_F\cdot[\psi]=[\psi^*]$, and it follows that 
$\overline{\Theta_\psi}=\Theta_{\sigma_F\psi}$. Taking sums over 
a set of representatives of optimal embeddings shows that 
$\overline{\mathbb{L}(f,\chi)}=\chi(\sigma_F)\cdot\mathbb{L}(f,\chi)$.
Let $H_\chi$ be the field cut out by $\chi$. The description of 
a system of representatives $\Sigma_c$ and $\Sigma_c^+$ of 
$\Gal(H_c/F)$ and $\Gal(H_c^+/F)$ in \S\ref{special-values} shows that 
if $\chi(\sigma_F)=1$ then $H_\chi$ 
is contained in $H_c$, and therefore $H_\chi$ is totally real. On the other hand, 
if $\chi(\sigma_F)=-1$, then $H_\chi$ cannot be contained in $H_c$, and therefore 
it is not totally real, so it is the product of two imaginary extensions. 
By definition of the sign of $\chi$, this means that 
$\mathbb{L}(f,\chi)$ is a real number when $\chi$ is even, 
and is a purely imaginary complex number when $\chi$ is odd, 
and the result follows.   
\end{proof}

Using the relation 
\[L(\pi_g\times\pi_\chi,1/2) = \frac 4{(2\pi)^k}\left(\left (\frac{k-2}{2}\right)!\right)^2 L(f/F,\chi,k/2),\] 
it follows from Lemma \ref{lemmasign} that 
Theorem \ref{thm:sv} can be rewritten in the following form: 
\begin{equation}\label{SV} 
L(f/F,\chi,k/2)=\frac{
(2\pi i)^{k-2}\cdot\Omega_{f}^2{\cdot w_\infty}}
{\left(\left(\frac{k-2}{2}\right)!\right)^2 \cdot \alpha^2 \cdot(Dc^2)^{(k-1)/2}
}\cdot\mathbb{L}(f,\chi)^2.
\end{equation} 

\begin{remark} By the lemma, the sign $w_\infty$ should also appear in equation (28) of 
\cite{BD-Rationality},
as the left hand side of that equation is not positive when $\chi$ is odd.  However, the
main result in \cite{BD-Rationality} still follows as this sign will cancel out with a sign arising from 
Gauss sums as in our argument below.
\end{remark}

\section{$p$-adic $L$-functions}

Recall the notation introduced in \S \ref{sec:2.4}:
$f_\infty$ is the Hida family
passing through the weight two modular form $f$ of level $N=Mp$ 
associated to the elliptic curve $E$ by modularity; $U$ is a connected 
neighborhood of $2$ in the weight space $\mathcal{X}$; $\mu_*^\pm$ 
is a measure-valued modular symbol satisfying the property that 
for all integers $k\in U$, $k\geq 2$, there is $\lambda^\pm(k)\in\C_p^\times$
such that $\rho_k(\mu_*^\pm)=\lambda^\pm(k)I_{f_k}^\pm$ and $\lambda^\pm(2)=1$.

\subsection{$p$-adic $L$-function of real quadratic fields} 
For any $Q\in \mathcal{F}_{Dc^2}$ and $\kappa\in U$,
define 
\[Q(x,y)^{(\kappa-2)/2}=\exp_p\left(\frac{\kappa-2}{2}\log_q\left(\langle Q(x,y)\rangle\right)\right)\]
where $\exp_p$ is the $p$-adic exponential and for $x\in\Q_p$, we let $\langle x\rangle$ denote the principal unit of $x$, satisfying 
$x=p^{\ord_p(x)}\zeta\langle x\rangle$ for a $(p-1)$-th root of unity $\zeta$. 
Recall the Hida family $f_\infty$ introduced in \S \ref{sec:2.4}.

\begin{definition}
Let $Q\in\mathcal{F}_{Dc^2}$ and let $\gamma_{\tau_Q}$ be the generator of the stabilizer 
of the root $\tau_Q$ of $Q(z,1)$, chosen as in Definition \ref{def:Darmon}.

\begin{enumerate} 
\item Let $r\in\mathbb{P}^1(\Q)$. The \emph{partial square root $p$-adic $L$-function} 
attached to $f_\infty$, a choice of sign $\pm$, and $Q$ is the function of $\kappa\in U$ defined by 
\[\mathcal{L}_p^\pm(f_\infty/F,Q,\kappa)=\int_{(\Z_p^2)'}Q(x,y)^{(\kappa-2)/2}\, d\mu_*^\pm\{r\rightarrow \gamma_{\tau_Q}(r)\}(x,y).\]
\item Let $\chi$ be a character of $G_c^+$. The  
\emph{square root $p$-adic $L$-function} attached to $f_\infty$, a choice of sign $\pm$, 
and $\chi$ is the function of $\kappa\in U$ defined by 
\[\mathcal{L}_p^\pm(f_\infty/F,\chi,\kappa)=\sum_{\sigma\in G_c^+}\chi^{-1}(\sigma)
\mathcal{L}_p^\pm(f_\infty/F,Q^\sigma,\kappa).\] 
\item The \emph{$p$-adic $L$-function} attached to $f_\infty$, the sign $\pm$, 
and $\chi$ is 
\[L_p^\pm(f_\infty/F,\chi,\kappa)=\left(\mathcal{L}_p^\pm(f_\infty/F,\chi,\kappa)\right)^2.\]
\end{enumerate}
\end{definition} 

Let $\chi:G_c^+\rightarrow\{\pm 1\}$ be a quadratic ring class character. 
Let $\epsilon$ be the sign of $\chi$ and set $w_\infty=\epsilon$. 
Denote $\mu_*=\mu_*^{w_\infty}$, 
$\Omega_{f_k}=\Omega^{w_\infty}_{f_k}$, 
$\lambda(k)=\lambda^{w_\infty}(k)$ and $L_p(f_\infty/F,\chi,k)=L_p^{w_\infty}(f_\infty/F,\chi,k)$. Recall the newform $f_k^\sharp$ 
whose $p$-stabilization is the weight $k$ specialization 
of the Hida family $f_\infty$ introduced in \S \ref{sec:2.4}.
Define the algebraic part of the central value of the $L$-function of 
the newform $f_k^\sharp$ twisted by $\chi$ to be 
\[L^\mathrm{alg}(f_k^\sharp/F,\chi,k/2)=
\frac{\left(\left(\frac{k-2}{2}\right)!\right)^2\sqrt{D}c}{(2\pi i)^{k-2}\cdot\Omega_{f_k^\sharp}^2}\cdot L(f_k^\sharp/K,\chi,k/2).\]

\begin{theorem} \label{p-adic-real}
For all integers $k\in U$, $k\geq 2$, we have 
\[L_p(f_\infty/F,\chi,k)={\lambda(k)^2\cdot{\alpha^{2}\cdot(1-a_p(k)^{-2}p^{k-2})^2}\cdot(Dc^2)^{(k-2)/2}}\cdot L^\mathrm{alg}(f_k^\sharp/F,\chi,k/2)\]
where the rational number $\alpha$ is defined in \eqref{alpha}. 
\end{theorem} 

\begin{proof} 
By definition, 
\begin{eqnarray*}
\mathcal{L}_p(f_\infty/F,Q,k)&=&\int_{(\Z_p^2)'}Q(x,y)^{(k-2)/2} \, d\mu_*\{r\rightarrow \gamma_{\tau_Q}(r)\}(x,y)\\
&=& \lambda(k)(1-a_p(k)^{-2}p^{k-2})I_{f_k^\sharp}\{r\rightarrow\gamma_{\tau_Q}(r)\}(Q^{(k-2)/2}),
\end{eqnarray*}
where the last equality follows from \cite[Proposition 2.4]{BD-Rationality} and therefore we get, 
in the notation of Section \ref{sec:34},
\[L_p(f_\infty/F,\chi,k)= 
\lambda(k)^2(1-a_p(k)^{-2}p^{k-2})^2\cdot\mathbb{L}(f_k^\sharp,\chi)^2.\]
Using \eqref{SV} gives the result. 
\end{proof}

\subsection{Mazur--Kitagawa $p$-adic $L$-functions} 
Let $\chi:(\Z/m\Z)^\times\rightarrow\{\pm1\}$ be a primitive quadratic Dirichlet character
of conductor $m$. Suppose that 
$\chi(-1)=(-1)^{(k-2)/2}w_\infty$ and put $\Omega_{f_k}=\Omega_{f_k}^{w_\infty}$, 
$\lambda(k)=\lambda^{w_\infty}(k)$ and $\mu_*=\mu_*^{w_\infty}$. 
For $k\in U$ a positive integer define 
\[L^\mathrm{alg}(f_k^\sharp,\chi,k/2)=\frac{\tau(\chi)((k-2)/2)!}{(-2\pi i)^{(k-2)/2}\Omega_{f^\sharp_k}}L(f_k^\sharp,\chi,k/2)\] 
as the \emph{algebraic part} of the central special value of $L(f_k^\sharp,\chi,s)$, where $\tau(\chi)=\sum_{a=1}^m\chi(a)e^{2\pi i a/m}$ denotes the Gauss sum of the character $\chi$.  
The Mazur--Kitagawa $p$-adic $L$-function is defined as
\[L_p(f_\infty,\chi,k,s)=\sum_{a=1}^m\chi(pa)\int_{\Z_p^\times\times\Z_p^\times}\left(x-\frac{pa}{m}y\right)^{s-1}y^{k-s-1} \, d\mu_*\{\infty\rightarrow pa/m\}\]
and satisfies the following \emph{interpolation formula}: for all integers $k\in U$ with $k\geq 2$ we have 
\begin{equation}\label{MK}
L_p(f_\infty,\chi,k,k/2)=\lambda(k)(1-\chi(p)a_p(k)^{-1}p^{(k-2)/2})^2L^\mathrm{alg}(f_k^\sharp,\chi,k/2).\end{equation}

\subsection{A factorization formula for genus characters} 
Let $\chi:G_c^+\rightarrow \{\pm1\}$ be a primitive character, 
and let $\chi_{D_1d}:\Q(\sqrt{D_1d})\rightarrow\{\pm1\}$ and 
$\chi_{D_2d}:\Q(\sqrt{D_2d})\rightarrow\{\pm1\}$ be the associated quadratic 
Dirichlet characters. 

\begin{theorem}\label{thmfactor} The following equality 
\[L_p(f_\infty/F,\chi,\kappa)=
\alpha^{2}\cdot(Dc^2)^{(\kappa-2)/2}
\cdot L_p(f_\infty,\chi_{D_1d},\kappa,\kappa/2)\cdot L_p(f_\infty,\chi_{D_2d},\kappa,\kappa/2) \]
holds for all $\kappa\in U$, where the rational number $\alpha$ is
defined in \eqref{alpha}.\end{theorem}

\begin{proof} Let $\chi_{D_i d}$ denote the quadratic characters 
associated with the extension $\Q(\sqrt{D_i d})$.
Since $p$ is inert in $F$, we have $\chi_D(p)=-1$, and
therefore from \eqref{characters} we get
\[\chi_{D_1d}(p)=-\chi_{D_2d}(p).\]
It follows that 
the Euler factor  $(1-a_p(k)^{-2}p^{k-2})^2$ appearing in Theorem \ref{p-adic-real} 
is equal to the product of the two Euler factors 
\[(1-\chi_{D_1d}(p)a_p(k)^{-1}p^{(k-2)/2})^2 \quad \text{and} \quad 
(1-\chi_{D_2d}(p)a_p(k)^{-1}p^{(k-2)/2})^2\]
appearing in \eqref{MK}. By comparison of Euler factors, we see that 
for all even integers $k\geq 4$ in $U$ we have 
\begin{equation}\label{FF}
L(f_k^\sharp/F,\chi,s)=L(f_k^\sharp,\chi_{D_1d},s)\cdot L(f_k^\sharp,\chi_{D_2d},s).\end{equation}
Therefore, from Theorem \ref{p-adic-real} and the factorization formula \eqref{FF} it follows that for all even integers $k\geq 4$ in $U$
the following factorization formula holds:  
\begin{equation}\label{eqfact}
L_p(f_\infty/F,\chi,k)=
\left(\frac{ \alpha^{2}\cdot \sqrt{D}c\cdot(Dc^2)^{(k-2)/2}{\cdot w_\infty}}{\tau(\chi_{D_1d})\cdot\tau(\chi_{D_2d})}\right)\cdot L_p(f_\infty,\chi_{D_1d},k,k/2)\cdot L_p(f_\infty,\chi_{D_2d},k,k/2).
\end{equation}
Since $D_id$ are fundamental discriminants, $\tau(\chi_{D_id})=\sqrt{D_id}$
(interpreting $\sqrt x$ as $i \sqrt{|x|}$ for $x < 0$),
so $\frac{\sqrt{D}c}{\tau(\chi_{D_1d})\cdot\tau(\chi_{D_2d})}={w_\infty}$, and the formula 
in the statement 
holds for all even integers $k\geq 4$ in $U$. 
Since $\Z\cap U$ is a dense subset of $U$, 
and the two sides of equation \eqref{eqfact} are continuous 
functions in $U$, they coincide on $U$. 
\end{proof}

\section{The main result}
Let the notation be as in the introduction: $E/\Q$ is an elliptic curve of conductor $N=Mp$ 
with $p\nmid M$, $p \ne 2$, and $F/\Q$ a real quadratic field of discriminant $D=D_F$ such that 
all primes dividing $M$ are split in $F$ and $p$ is inert in $F$. Finally, $c\in\Z$ is a positive 
integer prime to $ND$ and $\chi:G_c^+\rightarrow\{\pm\}$ is a primitive quadratic character of the strict 
ring class field of conductor $c$ of $F$. 
Let $w_\infty$ be the sign of $\chi$, and as above put $\mathcal{L}_p(f_\infty/F,Q,\kappa)=\mathcal{L}_p^{w_\infty}(f_\infty/F,Q, \kappa)$ and $L_p(f_\infty/F,\chi,\kappa)=L_p^{w_\infty}(f_\infty/F,\chi,\kappa)$.

We begin by observing that 
$\mathcal{L}_p(f_\infty/F,Q,2)=0$, since its value is $\mu_f\{r\rightarrow\gamma_{\tau_Q}(r)\}(\mathbb{P}^1(\Q_p))$, and the total measure of $\mu_f$ is zero. For the next result, 
let $w_M$ be the sign of the Atkin--Lehner involution acting on $f$. Also, let 
$\log_E:E(\C_p)\rightarrow\C$ denote the logarithmic map on $E(\C_p)$ induced from the 
Tate uniformization and the choice of the branch $\log_q$ of the logarithm fixed above. 

\begin{theorem}\label{final-1} 
For all quadratic characters $\chi:G_c^+\rightarrow\{\pm1\}$ we have 
\[\frac{d}{d\kappa}\mathcal{L}_p(f_\infty/F,\chi,\kappa)_{\kappa=2}=
\frac{1}{2}\left(1-\chi_{D_1d}(-M)w_M\right)\log_E(P_\chi), \]
where $P_\chi$ is defined as in \eqref{Pchi}.
\end{theorem}

\begin{proof}
We have  
\begin{eqnarray*}\frac{d}{d\kappa}\mathcal{L}_Q(f_\infty/F,\chi,\kappa)_{\kappa=2}&=&
\frac{1}{2}\int_{(\Z_p^2)'}(\log_q(x-\tau_Qy)+\log_q(x-\bar\tau_Qy)) \, d\mu_*\{r\rightarrow\gamma_{\tau_Q}(r)\}\\
&=&\frac{1}{2}(\log_E(P_{\tau_Q})+\log_E(\tau_pP_{\tau_Q})).
\end{eqnarray*}
By \eqref{real-conj}, 
$\tau_p(J_{\tau_Q})=-w_MJ_{\tau_Q^{\sigma_{\tau_Q}}}$ 
and by \cite[Proposition 1.8]{BD-Rationality} 
(whose extension to the present situation presents no difficulties) 
we know that $\chi(\sigma)=\chi_{D_1d}(-M)$, 
so the result follows summing over all $Q$.  
\end{proof}

\begin{theorem} Let $\chi$ be a primitive quadratic character of $G_c^+$ 
with associated Dirichlet characters $\chi_{D_1d}$ and $\chi_{D_2d}$. 
Suppose that $\chi_{D_1d}(-M)=-w_M$. Then: 
\begin{enumerate} 
\item There is a point $\mathbf{P}_\chi$ in $E(H_\chi)^\chi$ 
and $n\in\Q^\times$ such that $\log_E(P_\chi)=n\cdot\log_E(\mathbf{P}_\chi)$. 
\item The point $\mathbf{P}_\chi$ is of infinite order 
if and only if $L'(E/F,\chi,1)\neq 0$. 
\end{enumerate}\end{theorem}

\begin{proof}
By Theorem \ref{final-1} we have 
\[\frac{1}{2}\frac{d^2}{d\kappa^2}L_p(f_\infty/F,\chi,\kappa)_{\kappa=2}=\log^2_E(P_\chi). \]
On the other hand, by the factorization of Theorem \ref{thmfactor} we have 
\[L_p(f_\infty/F,\chi,\kappa)=
\alpha^{2}\cdot(Dc^2)^{(k-2)/2}
\cdot L_p(f_\infty,\chi_{D_1d},\kappa,\kappa/2)\cdot L_p(f_\infty,\chi_{D_2d},\kappa,\kappa/2), \]
where the integer $\alpha$ is defined in \eqref{alpha}. 
Let $\mathrm{sign}(E,\chi_{D_id})=-w_N\chi_{D_id}(-N)$, where $w_N$ is the sign 
of the Atkin--Lehner involution at $N$. This is the sign of the functional equation of the complex 
$L$-series $L(E,\chi_{D_id},s)$. Since \[\chi_{D_1d}(-N)\cdot\chi_{D_2d}(-N)=\chi_D(-N)=-1, \]
we may order the characters $\chi_{D_1d}$ and $\chi_{D_2d}$ in such a way that
$\mathrm{sign}(E,\chi_{D_1d})=-1$ and 
$\mathrm{sign}(E,\chi_{D_2d})=+1$. So $\chi_{D_1d}(-N)=w_N$ and since 
$\chi_{D_1d}(-M)=-w_M$ it follows that $\chi_{D_1d}(p)=-w_p=a_p$. So the Mazur--Kitagawa $p$-adic $L$-function $L_p(f,\chi_{D_1d},\kappa,s)$ has an exceptional zero at $(\kappa,s)=(2,1)$, 
and its order of vanishing is at least $2$. 
We may apply \cite[Theorem 5.4]{BD-HidaFamilies}, \cite[Sec. 6]{Mok} and \cite[Theorem 3.1]{Mok1}, 
which show that there is a global point $\mathbf{P}_{\chi_{D_1d}}\in E(\Q(\sqrt{D_1c}))$ and a 
rational number $\ell_1\in\Q^\times$ such that 
\[\frac{d^2}{d\kappa^2}L_p(f_\infty,\chi_{D_1d},\kappa,\kappa/2)_{\kappa=2}=\ell_1\log_E^2(\mathbf{P}_{\chi_{D_1d}}),\] 
and this point is of infinite order if and only if $L'(E,\chi_{D_1d},1)\neq 0$. 
Moreover, $\ell_1\equiv L^\mathrm{alg}(f,\psi,1)$ mod $(\Q^\times)^2$ for any primitive Dirichler 
character $\psi$ such that $L(f,\psi,1)\neq 0$, $\psi(p)=-\chi_{D_1d}(p)$, 
and $\psi(\ell)=\chi_{D_1d}(\ell)$ for all  $\ell\mid M$. 
Now 
\[\ell_2=\frac{1}{2}L_p(f_\infty,\chi_{D_2d},2,1)=L^\mathrm{alg}(E,\chi_{D_2d},1)\] 
is a rational number which is non-zero if and only if $L(E,\chi_{D_2d},1)\neq 0$. 
In this case, $\ell_1\ell_2$ is a square: choose $t\in\Q^\times$ such that 
$t^2=\ell_1\ell_2$ if $\ell_2\neq 0$ and $t=1$ otherwise, and let 
$\mathbf{P}_\chi=\mathbf{P}_{\chi_{D_1d}}$ in the first case 
and 0 otherwise. Now the first part of the theorem follows 
setting $n=\alpha \cdot t$. 
Finally, for the second part note that  $L(E,\chi_{D_2d},1)\neq 0$ if and only if $L'(E/F,\chi,1)\neq0$
thanks to the factorization \eqref{FF}.
\end{proof} 

\bibliographystyle{amsalpha}
\bibliography{paper}

\providecommand{\bysame}{\leavevmode\hbox to3em{\hrulefill}\thinspace}
\providecommand{\MR}{\relax\ifhmode\unskip\space\fi MR }
\providecommand{\MRhref}[2]{%
  \href{http://www.ams.org/mathscinet-getitem?mr=#1}{#2}
}
\providecommand{\href}[2]{#2}
\begin{thebibliography}{GMcS15}

\bibitem[BD07]{BD-HidaFamilies}
Massimo Bertolini and Henri Darmon, \emph{Hida families and rational points on
  elliptic curves}, Invent. Math. \textbf{168} (2007), no.~2, 371--431.
  \MR{2289868}

\bibitem[BD09]{BD-Rationality}
\bysame, \emph{The rationality of {S}tark--{H}eegner points over genus fields
  of real quadratic fields}, Ann. of Math. (2) \textbf{170} (2009), no.~1,
  343--370. \MR{2521118}

\bibitem[BSV17]{BSV}
Massimo Bertolini, Marco~Adamo Seveso, and Rodolfo Venerucci,
  \emph{{R}eciprocity laws for diagonal classes and rational points on elliptic
  curves}, preprint (2017).

\bibitem[Coh78]{cohn}
Harvey Cohn, \emph{A classical invitation to algebraic numbers and class
  fields}, Springer-Verlag, New York-Heidelberg, 1978, With two appendices by
  Olga Taussky: ``Artin's 1932 G\"ottingen lectures on class field theory'' and
  ``Connections between algebraic number theory and integral matrices'',
  Universitext. \MR{506156}

\bibitem[Dar01]{Dar}
Henri Darmon, \emph{Integration on {$\mathscr H_p\times\mathscr H$} and
  arithmetic applications}, Ann. of Math. (2) \textbf{154} (2001), no.~3,
  589--639. \MR{1884617}

\bibitem[Das05]{Das}
Samit Dasgupta, \emph{Stark--{H}eegner points on modular {J}acobians}, Ann.
  Sci. \'Ecole Norm. Sup. (4) \textbf{38} (2005), no.~3, 427--469. \MR{2166341}

\bibitem[DR17]{DR}
Henri Darmon and Victor Rotger, \emph{Stark-{H}eegner points and generalized
  {K}ato classes}, preprint (2017).

\bibitem[DT08]{DT}
Henri Darmon and Gonzalo Tornar\'{i}a, \emph{Stark-{H}eegner points and the
  {S}himura correspondence}, Compos. Math. \textbf{144} (2008), no.~5,
  1155--1175. \MR{2457522}

\bibitem[FMP17]{FMP}
Daniel File, Kimball Martin, and Ameya Pitale, \emph{Test vectors and central
  {$L$}-values for {${\rm GL}(2)$}}, Algebra Number Theory \textbf{11} (2017),
  no.~2, 253--318. \MR{3641876}

\bibitem[GKZ87]{GKZ}
B.~Gross, W.~Kohnen, and D.~Zagier, \emph{Heegner points and derivatives of
  {$L$}-series. {II}}, Math. Ann. \textbf{278} (1987), no.~1-4, 497--562.
  \MR{909238}

\bibitem[GM13]{GM5}
Xavier Guitart and Marc Masdeu, \emph{Computation of {ATR} {D}armon points on
  nongeometrically modular elliptic curves}, Exp. Math. \textbf{22} (2013),
  no.~1, 85--98. \MR{3038785}

\bibitem[GM14]{GM4}
\bysame, \emph{Overconvergent cohomology and quaternionic {D}armon points}, J.
  Lond. Math. Soc. (2) \textbf{90} (2014), no.~2, 495--524. \MR{3263962}

\bibitem[GM15a]{GM3}
\bysame, \emph{Elementary matrix decomposition and the computation of {D}armon
  points with higher conductor}, Math. Comp. \textbf{84} (2015), no.~292,
  875--893. \MR{3290967}

\bibitem[GM15b]{GM2}
\bysame, \emph{A {$p$}-adic construction of {ATR} points on
  {$\Bbb{Q}$}-curves}, Publ. Mat. \textbf{59} (2015), no.~2, 511--545.
  \MR{3374616}

\bibitem[GMcS15]{GM1}
Xavier Guitart, Marc Masdeu, and Mehmet~Haluk \c~Seng\"un, \emph{Darmon points
  on elliptic curves over number fields of arbitrary signature}, Proc. Lond.
  Math. Soc. (3) \textbf{111} (2015), no.~2, 484--518. \MR{3384519}

\bibitem[Gre09]{MG}
Matthew Greenberg, \emph{Stark--{H}eegner points and the cohomology of
  quaternionic {S}himura varieties}, Duke Math. J. \textbf{147} (2009), no.~3,
  541--575. \MR{2510743}

\bibitem[GS16]{GS}
Matthew Greenberg and Marco Seveso, \emph{{$p$}-adic families of cohomological
  modular forms for indefinite quaternion algebras and the
  {J}acquet-{L}anglands correspondence}, Canad. J. Math. \textbf{68} (2016),
  no.~5, 961--998. \MR{3536925}

\bibitem[GSS16]{GSS}
Matthew Greenberg, Marco~Adamo Seveso, and Shahab Shahabi, \emph{Modular
  {$p$}-adic {$L$}-functions attached to real quadratic fields and arithmetic
  applications}, J. Reine Angew. Math. \textbf{721} (2016), 167--231.
  \MR{3574881}

\bibitem[LM18]{LZ}
Matteo Longo and Zhengyu Mao, \emph{Kohnen's formula and a conjecture of
  {D}armon and {T}ornar\'{i}a}, Trans. Amer. Math. Soc. \textbf{370} (2018),
  no.~1, 73--98. \MR{3717975}

\bibitem[LN13]{LN}
Matteo Longo and Marc-Hubert Nicole, \emph{The {S}aito-{K}urokawa lifting and
  {D}armon points}, Math. Ann. \textbf{356} (2013), no.~2, 469--486.
  \MR{3048604}

\bibitem[LRV12]{LRV1}
Matteo Longo, Victor Rotger, and Stefano Vigni, \emph{On rigid analytic
  uniformizations of {J}acobians of {S}himura curves}, Amer. J. Math.
  \textbf{134} (2012), no.~5, 1197--1246. \MR{2975234}

\bibitem[LRV13]{LRV}
\bysame, \emph{Special values of {$L$}-functions and the arithmetic of {D}armon
  points}, J. Reine Angew. Math. \textbf{684} (2013), 199--244. \MR{3181561}

\bibitem[LV11]{LV-MM}
Matteo Longo and Stefano Vigni, \emph{Quaternion algebras, {H}eegner points and
  the arithmetic of {H}ida families}, Manuscripta Math. \textbf{135} (2011),
  no.~3-4, 273--328. \MR{2813438}

\bibitem[LV14]{LV}
\bysame, \emph{The rationality of quaternionic {D}armon points over genus
  fields of real quadratic fields}, Int. Math. Res. Not. IMRN (2014), no.~13,
  3632--3691. \MR{3229764}

\bibitem[LV16]{LV1}
\bysame, \emph{Quaternionic {D}armon points on abelian varieties}, Riv. Math.
  Univ. Parma (N.S.) \textbf{7} (2016), no.~1, 39--70. \MR{3675402}

\bibitem[Mok]{Mok1}
Chung~Pang Mok, \emph{On a theorem of {B}ertolini--{D}armon about rationality
  of {S}tark--{H}eegner points over genus fields of real quadratic fields},
  preprint.

\bibitem[Mok11]{Mok}
\bysame, \emph{Heegner points and {$p$}-adic {$L$}-functions for elliptic
  curves over certain totally real fields}, Comment. Math. Helv. \textbf{86}
  (2011), no.~4, 867--945. \MR{2851872}

\bibitem[MW09]{MW}
Kimball Martin and David Whitehouse, \emph{Central {$L$}-values and toric
  periods for {${\rm GL}(2)$}}, Int. Math. Res. Not. IMRN (2009), no.~1, Art.
  ID rnn127, 141--191. \MR{2471298}

\bibitem[Pop06]{Popa}
Alexandru~A. Popa, \emph{Central values of {R}ankin {$L$}-series over real
  quadratic fields}, Compos. Math. \textbf{142} (2006), no.~4, 811--866.
  \MR{2249532}

\bibitem[RS12]{RS}
Victor Rotger and Marco~Adamo Seveso, \emph{{$\mathscr L$}-invariants and
  {D}armon cycles attached to modular forms}, J. Eur. Math. Soc. (JEMS)
  \textbf{14} (2012), no.~6, 1955--1999. \MR{2984593}

\bibitem[Sev12]{S}
Marco~Adamo Seveso, \emph{{$p$}-adic {$L$}-functions and the rationality of
  {D}armon cycles}, Canad. J. Math. \textbf{64} (2012), no.~5, 1122--1181.
  \MR{2979580}

\bibitem[Tri06]{Trif}
Mak Trifkovi\'c, \emph{Stark--{H}eegner points on elliptic curves defined over
  imaginary quadratic fields}, Duke Math. J. \textbf{135} (2006), no.~3,
  415--453. \MR{2272972}

\bibitem[Ven18]{Guhan}
Guhan Venkat, \emph{Darmon cycles and the {K}ohnen-{S}hintani lifting}, Trans.
  Amer. Math. Soc. \textbf{370} (2018), no.~6, 4059--4087. \MR{3811520}

\bibitem[Vig80]{Vigneras}
Marie-France Vign\'{e}ras, \emph{Arithm\'{e}tique des alg{\`e}bres de
  quaternions}, Lecture Notes in Mathematics, vol. 800, Springer, Berlin, 1980.
  \MR{580949}

\end{thebibliography}
\end{document}